\documentclass[11pt]{article}

\usepackage[utf8]{inputenc}
\usepackage[vietnamese,english]{babel}
\usepackage[displaymath, tightpage]{preview}
\usepackage[all, cmtip]{xy}

\usepackage{amsmath}
\usepackage{amsxtra}
\usepackage{amscd}
\usepackage{amsthm}
\usepackage{amsfonts}
\usepackage{amssymb}
\usepackage{bbm}

\usepackage[bitstream-charter]{mathdesign}
\usepackage[T1]{fontenc}
\usepackage{fancyhdr}
\usepackage{tikz}
\usepackage{tikz-cd}
\usetikzlibrary{matrix,calc,arrows}

\newcommand\nbc{\foreignlanguage{vietnamese}{Ngô Bảo Châu}}

\numberwithin{equation}{section}

\newtheorem{theorem}{Theorem}
\newtheorem{lemma}[theorem]{Lemma}
\newtheorem{proposition}[theorem]{Proposition}
\newtheorem{definition}[theorem]{Definition}
\newtheorem{corollary}[theorem]{Corollary}

\numberwithin{theorem}{section}

\theoremstyle{definition}

\newcommand\plim{{\underset{\longleftarrow}{\lim}}\ }
\newcommand\ilim{{\underset{\longrightarrow}{\lim}}\ }

\newcommand\indlim\varinjlim

\newcommand*{\ehat}{{\mathchoice{%
  \hbox{[do not use `ehat` in display style]}%
  }{%
  \hbox{[do not use `ehat` in text style]}%
  }{%
  \mbox{\raisebox{-0.75\height}{$\hat{\mkern4mu}$}}%
  }{%
  \mbox{\raisebox{-0.75\height}{$\scriptstyle\hat{\mkern4mu}$}}%
  }}}%

\newcommand\cL{\mathcal{L}}

\newcommand\cO{\mathcal{O}}

\newcommand\fp{\mathfrak{p}}
\newcommand\fq{\mathfrak{q}}

\newcommand\fm{\mathfrak{m}}

\newcommand\bbA{\mathbb{A}}

\newcommand\NN{\mathbb{N}}

\newcommand\bA{{\bbA}}

\newcommand\Z{\mathbb{Z}}
\newcommand\Q{\mathbb{Q}}

\newcommand\C{\mathbb{C}}
\newcommand\A{\mathbb{A}}
\newcommand\N{\mathbb{N}}

\def\d{\mathrm{d}}

\newcommand\rT{\mathrm{T}}

\newcommand\Jac{\mathrm{Jac}}
\newcommand\pr{\mathrm{pr}}

\newcommand\id{\mathrm{id}}
\newcommand\Spec{\mathrm{Spec}}
\newcommand\Spm{\mathrm{Spm}}

\newcommand\Gm{\mathbb{G}_m}

\newcommand\val{\mathrm{val}}

\newcommand\im{\mathrm{im}}

\newcommand\Mat{\mathrm{Mat}}

\newcommand\ord{\mathrm{ord}}

\newcommand\height{\mathrm{ht}}

\newcommand\nil{\mathrm{nil}}

\newcommand{\IC}{\mathrm{IC}}

\newcommand\Ql{{\Q}_\ell}

\usepackage{natbib}
\title{Weierstrass preparation theorem and singularities in the space of non-degenerate arcs}
\author{\nbc}
\date{}
\begin{document}
\maketitle

\section{Introduction}
It has been long expected that there exists a deep connection between singularities of certain arc spaces and harmonic analysis over nonarchimedean fields. For instance, certain functions appearing naturally in harmonic analysis can be interpreted as the function attached to the trace of the Frobenius operators on what should be the stalks of the intersection complex of certain arc spaces, see \cite{Bouthier:2014vp}. However, a proper foundation of a theory of perverse sheaves on arc spaces is still missing even though a recent work of Bouthier and Kazhdan \cite{Bouthier:2015tk} outlines a strategy for setting it up.

As the theory of perverse sheaves is originally built for schemes of finite type, the basic difficulty in extending it to arc spaces is that those spaces are almost always infinite dimensional. The first inroad into this new territory is made by Grinberg and Kazhdan who prove that the formal completion of the arc spaces at a point representing a non-degenerate arc is isomorphic to the formal completion of a scheme of finite type, augmented by infinitely many free formal variables, under the assumption that the base field is the field of complex numbers. This result is later improved by Drinfeld who prove it over an arbitrary base field. 

Let us fix the notations in order to state Grinberg-Kazhdan-Drinfeld's theorem. Let $k$ be a field. Let $X$ be an affine $k$-scheme of finite type. For every $n\in\N$, we consider the space of $n$th jets on $X$ representing the functor $R\mapsto \cL_n X(R)=X(R[t]/t^{n+1})$ on the categories of $k$-algebras. The arc space of $X$ is the limit of $\cL_nX$ as $n\to \infty$:
\begin{equation}
	\cL X(R)= \plim X(R[t]/t^{n+1}) = X(R[[t]]).
\end{equation}   
If $X=\Spec(A_0)$ is an affine $k$-scheme, for every $n\in\N$, $\cL_n X$ is represented by an affine $k$-scheme $\cL_n X=\Spec(A_n)$ then $\cL X= \Spec(A)$ where $A=\ilim A_n$ is the filtered colimit of $A_n$ as ${n\to\infty}$. 

If $X$ is a smooth, then the space $\cL_n X$ of $n$-th jets on $X$ is also smooth. For every $n\in\N$, the transition morphism $\cL_{n+1} X\to \cL_n X$ is smooth and surjective. More precisely, the transition morphism $\cL_{n+1} X\to \cL_n X$ is a torsor under certain vector bundles over $\cL_n X$ of rank equal the dimension of $X$.  If $X$ is not smooth, the situation is much more complicated: the transition morphism $\cL_{n+1} X\to \cL_n X$ is neither smooth nor surjective.  

Let $X'$ be a smooth open subscheme of $X$ and $Z$ a closed subscheme of $X$ complement of $X'$.
We are mainly interested in the open subscheme of non-degenerate arcs whose $k$-points form the set
$$\cL^\bullet X(k)=\cL X(k)-\cL Z(k). $$ 
The definition of non-degenerate arc space depends thus on the choice of a smooth open subscheme $X'$ of $X$. Although we may of take $X'$ to be the smooth locus of $X$, it is often more convenient to keep the possibility of choosing $X'$ smaller than the smooth locus of $X$. For instance, the case of the affine line $X=\bA^1$ and $X'=\Gm$ will be of special interest for the study of the Weierstrass preparation theorem. 
Some cares are in order to enunciate the functorial description of $\cL^\bullet X$.

\begin{definition}
The non-degenerate arc space of $X$ relative to the smooth subscheme $X'$ is the functor on the category of $k$-algebras which attaches to each $k$-algebra $R$ the set $\cL^\bullet X(R)$ consisting of maps $x: \Spec (R[[t]]) \to X$ such that the projection $x^{-1}(X')\to \Spec(R)$ is surjective.	
\end{definition}

We can now state the theorem of Grinberg-Kazhdan and Drinfeld, \cite{Grinberg:1998tv}, \cite{Drinfeld:2002vda}.

\begin{theorem}
\label{Grinberg-Kazhdan}
	Let $x\in \cL X(k)$ be a point $x:\Spec (k[[t]]) \to X$ such that the
restriction to $\Spec(k((t)))$ has image in $X'$. Then the completion of $\cL
X$ at $x$ has a finite type formal model model i.e. there exists a $k$-scheme $Y$ of finite type with a point $y\in Y(k)$ such that there exists an isomorphism 
	$$ (\cL X)^\ehat _x \simeq \hat Y_y \hat\times \hat D^\infty$$
where $\hat D^\infty={\rm Spf} (k[[x_1,x_2,\ldots]])$.
\end{theorem}

Grinberg-Kazhdan-Drinfeld's theorem gives the hope that a reasonable theory of perverse sheaves on formal arc spaces may exist. In \cite{Bouthier:2014vp}, it is proven that in a weak sense the formal finite dimensional model is independent of choices i.e if $\hat Y_y$ and $\hat Y'_{y'}$ are finite type formal models of $\hat X_x$
$$ (\cL X)_x \simeq \hat Y_y \hat\times \hat D^\infty \simeq \hat Y'_{y'}\hat\times \hat D^\infty$$
then there exists integers $m,n\in\N$ such that there exists an isomorphism
\begin{equation} \label{finite-equivalence}
	\hat Y_y \hat\times D^m \simeq \hat Y'_{y'}\hat\times D^n.
\end{equation}
As consequence in the case $k$ is a finite field, using the $\ell$-adic intersection complex on a scheme $Y$ of finite type of which $\hat Y_y$ is formal finite dimensional model of $\hat X_x$, we can define a canonical function $\IC_{\cL X}:\cL^\bullet X(k)\to\Ql$ which deserves the name of intersection complex function although we don't know yet to define the intersection complex on the space of non-degenerate arcs. If $x\in \cL X(k)$ is a non-degerate arc on $X$, and $\hat Y_y$ is a finite-dimensional formal as in \ref{Grinberg-Kazhdan}, then we set 
$$\IC_{\cL X}(x)=\IC_Y(y).$$
The existence of the isomorphism \eqref{finite-equivalence} implies that $\IC_Y(y)=\IC_{Y'}(y')$ so that this number doesn't depend on the choice of the finite type formal model as argued in \cite[Prop. 1.2]{Bouthier:2014vp}.

In order to define perverse sheaves on the space of non-degenerate arcs, one may hope a stronger version of the Drinfeld-Grinberg-Kazhdan theorem in which formal completions are replaced by strict Henselizations for instance. In other words, instead of formal charts as in \ref{Grinberg-Kazhdan}, one would like to construct Henselian charts. 
There are good reasons to believe that the analogue of Grinberg-Kazhdan theorem for henselizations doesn't hold. 

In \cite{Bouthier:2015tk}, Bouthier and Kazhdan attempt to construct certain type of coverings of the space of non-degenerate arcs which admit essentially smooth surjective map to schemes of finite type. The strategy of \cite{Bouthier:2015tk} consists in generalizing Drinfeld's construction in \cite{Drinfeld:2002vda}, which works over local Artinian test rings, to general test rings. There are unfortunately some gaps in the construction of \cite{Bouthier:2015tk}. 

Both constructions \cite{Drinfeld:2002vda} and \cite{Bouthier:2015tk} rely in the Weierstrass preparation theorem and the Newton method of solving algebraic equation by approximation. We will study with some care the Weierstrass preparation and division theorem, and the Newton method over a general test rings. The study of the Weierstrass division theorem reveals that the coordinate ring of the arc space of the curve of equation $xy=0$ contains non-zero functions which vanish at all points to the infinite order. We observe that these functions are annihilated in completed local rings but generally not in henselizations. We also observe that if $X$ is a $k$-scheme of finite type, henselizations of the ring of coordinates $X\times \bA^\infty$ don't contain non-zero elements which vanish at all points to the infinite order.

Nevertheless, following the method of \cite{Drinfeld:2002vda} and \cite{Bouthier:2015tk},
one can obtain a theorem on local structure of non-degenerate arc space. Although the description of the local structure is not very pleasant, one can derive from it the following statement.

\begin{theorem}\label{slice} Let $X$ be a $k$-scheme locally of complete intersection.
	Let $x\in \cL X(k)$ be a non-degenerate arc. There exist a $k$-scheme $Y$ of finite type, a point $y\in Y(k)$, a morphism $\phi:Y\to \cL X$ mapping $y$ to $x$ such that for every field $k'$ containing $k$, for every point $y'\in Y(k')$ mapping to $x'\in \cL X(k')$, there exists an isomorphism 
	$$\cL X^\ehat _{x'} \hat\times \hat D^\infty \simeq \hat Y_{y'} \hat\times \hat D^\infty.$$ 
\end{theorem}

Schemes of finite type as in the above statement should be thought of as good slices in non-degenerate arc spaces. One may hope to construct a reasonable theory of perverse sheaves on arc spaces using these slices. For instance, one may think of the hypothetical intersection complex $\IC_{\cL X}$ of $\cL X$ as the unique object whose restriction to every slice $Y$ of finite type is the intersection complex $\IC_Y$ of $Y$ up to a shift by the dimension.

Although it may be possible to drop the complete intersection assumption, in this paper, I want, as much as possible, to stick  with the computational approach of \cite{Drinfeld:2002vda}. One should also note that in many moduli problems, slices are given to us by replacing maps from the formal disc to a given target scheme, or algebraic stack, by maps from a given curve. This idea is in fact behind the calculations in the complete intersection case as in \cite{Drinfeld:2002vda}. Drinfeld explains how one can approach the general case by means of the Newton groupoid in the Geometric Langlands Seminar in Spring 2017.

\section{Formal completion and tensor product}

There is a potential danger in dealing with formal series with coefficients in a varying commutative ring $R$ for the operation of tensor product doesn't commute with the operation of completion. The purpose of this section is to provide warnings by means of some examples where what we observe isn't quite what we would naively expect. Our examples are all related to the problem of Weierstrass division. 

For a commutative ring $R$, an ideal $I$ of $R$, and an $R$-module $M$, we define the $I$-completion of $M$ to be the limit 
\begin{equation*}
	\hat M_I=\plim M/I^n M.
\end{equation*}
For instance, the $p$-completion of the ring $\Z$ of integers is the ring $\Z_p$ of $p$-adic integers. At this place, it may be necessaty to clear some confusion in terminology: although the field $\Q_p$ of $p$-adic numbers, defined as the quotient field of $\Z_p$, is $p$-adically complete in the sense that all Cauchy sequences in $\Q_p$ for the $p$-adic distance converge, the $p$-completion of $\Q_p$ is zero.

For every commutative ring $R$, the ring of formal series $R[[t]]$ consists in set of  all series of the form $x=x_0+x_1 t + x_2 t^2 + \cdots$ where $x_0,x_1,\ldots$ are arbitrary elements subjected to usual rules of addition and multiplication. It contains the ring $R[t]$ of polynomials consisting of $x=x_0+x_1 t + x_2 t^2 + \cdots$ with coefficients $x_0,x_1,\ldots$ equal to $0$ except for finitely many of them. We may also define $R[[t]]$ as the $t$-completion of $R[t]$:
\begin{equation*}
	R[[t]]=\plim R[t]/(t^n).
\end{equation*} 

As unproblematic as this definition may appear, some good cares are in order. If $R$ is a commutative ring $R$, $I$ is an ideal of $R$, $M$ is a $R$-module, for an $R$-algebra $R'$, the canonical morphism
\begin{equation}\label{tensor-completion}
	\hat M_I \otimes_R R' \to (M\otimes_R R')^\ehat_{IR'}
\end{equation}
may not be an isomorphism.

For instance, $\Q[[t]] \otimes_\Q \C$ consists of finite complex linear combinations of formal series with rational coefficients while $\C[[t]]$ consists of all formal series with complex coefficients. In this example the map \eqref{tensor-completion} is injective but not surjective.

Now we consider the ring $R=\Z[t]$ of polynomial with integers coefficients, the ideal $I$ of $R$ generated by $t$ and the module 
\begin{equation}\label{bizzard}
	M=\Z[t]/(t-p).
\end{equation}
As an abelian group, $R$ is canonically isomorphic with $\Z$. In fact, we may see $R$ as the $\Z[t]$-module $\Z$ on which $t$ acts as the multiplication by $p$. The $t$-completion of $M$  is thus $\hat M_I=\Z_p$. Now if $R'=\Q[t]$, then on the one hand we have $\hat M_I\otimes_R R'=\Q_p$, but on the other hand the $t$-completion of $\Q[t]/(t-p)$ is zero as $p\in\Q^\times$. In particular, in this example the map \eqref{tensor-completion} is surjective but not injective.

The following lemma won't be used in the sequel except for a side observation, but it may serve as yet another warning. 

\begin{lemma} \label{base-change-lemma}
Let $R$ is a commutative ring $R$, $I$ a finitely generated ideal of $R$ and $\bar R=R/I$. Then the natural map $\phi:R[[t]]\otimes_R \bar R \to \bar R[[t]]$ is an isomorphism. The assertion doesn't hold without the assumption $I$ finitely generated
\end{lemma}

\begin{proof}
	Since $R\to \bar R$ is surjective, the map $R[[t]]\to \bar R[[t]]$ is surjective and so is $\phi:R[[t]]\otimes_R \bar R \to \bar R[[t]]$. We only need to prove that $\phi$ is injective. The injectivity of $\phi$ is equivalent to that the kernel of $R[[t]]\to \bar R[[t]]$ is the ideal of $R[[t]]$ generated by $I$. Let us denote $I[[t]]$ the kernel of $R[[t]] \to \bar R[[t]]$. This kernel consists in formal series $x=x_0+x_1 t+\cdots$ such that $x_n \in I$ for all $n$. If $a_1,\ldots,a_m$ is a system of generators of $I$, then for every $n\in\N$, the coefficient $x_n$ of $x$ can be written possibly non uniquely in the form $x_n=a_1 y_{n,1}+\cdots+a_m y_{n,m}$. If we choose such an expression for each $n\in\N$ then we have $x=a_1 y_1 +\cdots +a_m y_m$ where $y_i$ is the formal series $y_i=y_{i,0}+y_{i,1}t+\cdots$. It follows that $x$ belongs to the ideal of $R[[t]]$ generated by $I$.
	
	The assertion doesn't hold without the assumption that $I$ is finitely generated. Let $R=k[x_0,x_1,x_2,\ldots]$ be the ring of polynomials in infinitely many variables $x_0,x_1,x_2,\ldots$ with coefficients in a field $k$, $I$ the ideal of $R$ generated by $x_0,x_1,x_2,\ldots$, $\bar R=R/I=k$. Then the formal series $x_0+x_1t+x_2 t^2+\cdots$, lying in the kernel of $R[[t]]\to k[[t]]$, doesn't belong to  the ideal of $R[[t]]$ generated by $I$ but its $t$-completion.
\end{proof}

\section{Variations on the Weierstrass preparation theorem}

If $f$ is a nonzero germ of a holomorphic function at the origin in $\C^{n+1}$, given with coordinates $t,x_1,\ldots,x_n$, the classical Weierstrass preparation theorem asserts the existence of a germ of holomorphic function $u$ at the origin in $\C^{n+1}$ non vanishing at the origin and a polynomial
\begin{equation} \label{Weierstrass-polynomial}
	q=q_0 + q_1 t +\cdots + q_{d-1} t^{d-1} + t^d
\end{equation}
where $q_0,\ldots,q_{d-1}$ are germs of holomorphic functions at the origin $\C^n$, and vanishing at the origin, such that $f=uq$. 

Let $R_n=\C[[x_1,\ldots,x_n]]$ be the ring of formal series of variables $x_1,\ldots,x_n$. Let $f\in R_n[[t]]$ be a formal series of variables $t,x_1,\ldots,x_n$. We denote $\bar f\in \C[[t]]$ the reduction of $f$ modulo $x_1,\ldots,x_n$ and write $\bar f=t^d \bar u$ where $u\in\C[[t]]^\times$. By usual manipulations with formal series, we can prove that there exist a unique formal series $u\in R_n[[t]]^\times$ with reduction $\bar u\in \C[[t]]^\times$ and a unique polynomial $q\in R_n[t]$ of the form \eqref{Weierstrass-polynomial} whose coefficients $q_0,\ldots,q_{d-1}\in R_n$ are formal series vanishing constant coefficients such that $f=uq$. The holomorphic content of the Weierstrass preparation theorem consists in showing that if $f$ is absolutely convergent in a neighborhood of $0$ in $\C^{n+1}$ then $q_0,\ldots,q_{d-1}$ are also absolutely convergent in a neighborhood of $0$ in $\C^n$. 

Let $A_n$ denote the ring of germs of holomorphic functions defined in a neighborhood of $0$ in $\C^n$. $A_n$ is a Henselian local ring whose completion with respect to the maximal ideal is $R_n=\C[[x_1,\ldots,x_n]]$. Every $f\in A_n[[t]]$ can be factored uniquely in the form $f=uq$ with $u\in R_n[[t]]^\times$ and $q\in R_n[t]$ a Weierstrass polynomial as above but there is no guarantee that the coefficients $q_0,\ldots,q_{d-1}$ belong to $A_n$ unless $f$ is assumed to be a convergent series. We draw the conclusion that the existence of Weierstrass factorization for $f\in A[[t]]$ as $f=uq$ where $u\in A[[t]]^\times$ and $q\in A[t]$ is a Weierstrass polynomial doesn't hold generally for Henselian local ring $A$ as opposed to complete local rings. The aim of this section is to study the existence of the Weierstrass factorization over a general test ring and its connection with the geometry of the arc space of the affine line.

The arc space $\cL\bA^1$ of the affine line $\bA^1$ is spectrum of the ring of polynomials with countably infinite number of variables:
\begin{equation}
	\cL \bA^1=\Spec (k[x_0,x_1,x_2\ldots])
\end{equation}
for a $R$-point $x\in \cL \bA^1(R)$ is uniquely represented by a formal series
$x=x_0+x_1t+\cdots$
for every $k$-algebra $R$. 
We consider the space $\cL^\bullet\bA^1$ of non-degenerate arcs with respect to the open subset $\Gm$ of $\bA^1$. On the level of $k$-points we have
\begin{equation}
	\cL^\bullet\bA^1(k)=k[[t]]-\{0\}
\end{equation}
The description of set of $R$-points of $\cL^\bullet\bA^1$ is as follows:

\begin{definition} \label{non-degenerate-formal-series}
	For every commutative ring $R$, a formal series $x\in R[[t]]$ is said to be non-degenerate, and denoted $x\in \cL^\bullet\bA^1(R)$, if for every ring homomorphism $\nu_K:R\to K$ from $R$ to a field $K$ the formal series $\nu_K(x)$ is a nonzero element of $K[[t]]$, or in other words: $$\nu_K(x) \in K[[t]] \cap K((t))^\times.$$
\end{definition}

A strictly stronger condition is $x\in R[[t]] \cap R((t))^\times$ where $R((t))$ is the ring of Laurent formal series with coefficients in $R$. In other words, we have an inclusion 
\begin{equation} \label{first-inclusion}
	R[[t]]\cap R((t))^\times \subset \cL^\bullet \bA^1(R).
\end{equation}
The functor $R\mapsto R[[t]]\cap R((t))^\times$ is closely related to the affine Grassmannian of $\Gm$ defined as the the sheafification of the functor $R\mapsto R((t))^\times/ R[[t]]^\times$. Although our purpose is not to review the theory of affine Grassmannian, we will borrow some familiar arguments from that theory helping to clarify the relation between the properties of being invertible in $R((t))^\times$ and the existence of a strict Weierstrass factorization. 

\begin{definition}
	Let $R$ be a commutative ring. 
	
	\begin{enumerate}
		\item We will call a  Weierstrass factorization of a formal series $x\in R[[t]]$ an expression of the form $x=uq$ where $u\in R[[t]]^\times$ and $q$ is a monic polynomial of some degree $d$. 
		\item We will call the expression $x=uq$ as above a strict Weierstrass factorization of $x$ if $q$ divides $t^n$ for some $n\in\N$ i.e. $t^n \in (q)$ where $q$ is the ideal $(q)$ of $R[[t]]$ generated by $q$.
	\end{enumerate}
	\end{definition}

If we denote $\cL^\beta\bA^1(R)$ and $\cL^{\beta+}\bA^1(R)$ the set of formal series $x\in R[[t]]$ having a Weierstrass factorization and strict Weierstrass factorization respectively:
\begin{equation}
	\cL^{\beta+}\bA^1(R) \subset \cL^\beta\bA^1(R) \subset \cL^\bullet\bA^1(R).
\end{equation}
We will prove that the equality
\begin{equation}
	R[[t]]\cap R((t))^\times = \cL^{\beta+} \bA^1(R)
\end{equation}
holds for all commutative rings $R$.

\begin{proposition} \label{algebra}
For every formal series $x\in R[[t]]$, the following conditions are equivalent:
\begin{enumerate}
	\item $x\in R[[t]]\cap R((t))^\times$;
	\item $R[[t]]/(x) $ is a projective $R$-module of finite type, $(x)$ being the ideal of $R[[t]]$ generated by $x$, and $t^n\in (x)$ for some integer $n\in\NN$;
	\item $x$ admits a strict Weierstrass factorization: $x=uq$ where $u\in R[[t]]^\times$ and $q$ is a monic polynomial dividing $t^n$ for some $n\in\N$.
\end{enumerate}
If these conditions are satisfied, then the strict Weierstrass factorization of $x$ is unique.
\end{proposition}

\begin{proof}
First we prove that if $x\in R[[t]] \cap R((t))^\times$ then, for some $n\in\N$, $t^n$ belongs to the ideal $(x)$ of $R[[t]]$ generated by $x$. Let $y$ be the inverse of $x$ in $R((t))^\times$, we have $x y=1$. Since $y\in R((t))$  there exists $n\in\N$ such that $t^ny\in R[[t]]$. It follows that $t^n=x(y t^n) \in (x)$. It will also follow that $R[[t]]/x R[[t]]$ is a $R$-module of finite type as a quotient of the $R$-module of finite type $R[[t]]/t^n R[[t]]$.

Next we prove that if $x\in R[[t]] \cap R((t))^\times$ then $R[[t]]/x R[[t]]$ is a flat $R$-module. 	We consider the commutative diagram
\begin{equation}
\begin{tikzcd}
0 \arrow{r} & 
R[[t]] \arrow{r} \arrow{d}[swap]{x} & 
R((t)) \arrow{r} \arrow{d}[swap]{x} &
R((t))/R[[t]] \arrow{r} \arrow{d}[swap]{x} &
0\\
0 \arrow{r} & 
R[[t]] \arrow{r}  & 
R((t)) \arrow{r}  &
R((t))/R[[t]] \arrow{r}  &
0
\end{tikzcd}
\end{equation}
with exact horizontal lines. Since $x\in R((t))^\times$, the middle vertical map is an isomorphism. The snake lemma provides us with an exact sequence of $R$-modules
\begin{equation}
\begin{tikzcd}
0 \arrow{r} & 
R[[t]]/x R[[t]] \arrow{r}  & 
R((t))/R[[t]] \arrow{r}{x} &
R((t))/R[[t]] \arrow{r}  &
0\end{tikzcd}
\end{equation}
For $R((t))$ decomposes as a direct sum of $R$-modules
	$$R((t))=t^{-1} R[t^{-1}] \oplus R[[t]]$$
we can identify $R((t))/R[[t]]$ with $t^{-1} R[t^{-1}]$. It follows that $R((t))/R[[t]]$ is flat as $R$-module and thus $R[[t]]/x R[[t]]$ is also flat as $R$-module.

Now we prove that if $R[[t]]/x R[[t]]$ is a projective $R$-module of finite rank $d$ and if $t^n \in (x)$ for some $n\in\N$, then $x$ has a strict Weierstrass factorization. We consider the exact sequence of $R$-modules
\begin{equation} \label{q|t^n}
\begin{tikzcd}
0 \arrow{r} & 
(x)/(t^n)  \arrow{r}  & 
R[[t]]/(t^n)  \arrow{r} &
R[[t]]/(x) \arrow{r}  &
0\end{tikzcd}
\end{equation}
where $(x)$ and $(t^n)$ are ideals of $R[[t]]$ generated by $x$ and $t^n$ respectively. We know that $R((t))/t^n R[[t]]$ and $R((t))/xR[[t]]$ are flat $R$-modules of finite type. It follows that $(x)/(t^n) $ is a flat $R$-module. Being generated by $x, xt,\ldots,xt^{n-1}$, it is a $R$-module of finite type, and therefore $(x)/(t^n)$ is a projective module of finite type.

Let $q\in R[t]$ denote the characteristic polynomial of the $R$-linear operator $t$ acting on $R[[t]]/xR[[t]]$. By the Cayley-Hamilton theorem, the multiplication by $q$ in $R[[t]]/(x)$ is null. This implies that $q\in (x)$ or in other words $x$ divides $q$. 

We claim that $q$ divides $t^n$. Indeed if $q'$ is the characteristic polynomial of $t$ acting on the projective $R$-module of finite type $(x)/(t^n)$, then we have $t^n=qq'$ by the exact sequence \eqref{q|t^n}.

Now we consider morphisms of $R$-modules:
\begin{equation}\label{2-maps}
	R[t]/(q) \to R[[t]]/(q) \to R[[t]]/(x)
\end{equation}
where $R[t]/(q)$ is a free $R$-module of rank $R$ of rank $d$ since $q$ is a monic polynomial and $R[[t]]/(x)$ is also known to be projective $R$-modules of rank $d$ by assumption. Since $q$ divides $t^n$, $t$ acts nilpotently on $R[t]/q$. It follows that the map $R[t]/(q) \to R[[t]]/(q)$ is an isomorphism. It follows that $R[[t]]/q$ is also a free $R$-module of rank $d$. The surjective map $R[[t]]/(q) \to R[[t]]/(x)$ between projective $R$-modules of rank $d$ has to be an isomorphism. It follows that $x=qu$ where $u\in R[[t]]^\times$ and $q$ is a monic polynomial dividing $t^n$ thus $x$ has a strict Weierstrass factorization.

Now we assume that $x\in R[[t]]$ has a strict Weierstrass factorization $x=uq$ with $u\in R[[t]]^\times$ and $q$ is a monic polynomial dividing $t^n$ for some $n\in\N$ and we prove that $x\in R[[t]]\cap R((t))^\times$. In this case $q\in R((t))^\times$ and therefore $x\in R[[t]]\cap R((t))^\times$. 

Finally if $x=uq$ is a strict Weierstrass factorization then the linear maps \eqref{2-maps} are isomophisms of free $R$-modules. It follows that $q$ is the characteristic polynomial of $t$ acting on $R[[t]]/(x)$. Thus the strict Weierstrass factorization of $x$ is unique if it exists.
\end{proof}

\begin{corollary} \label{base-change}
	Let $x\in \cL^{\beta+}\bA^1(R)$ be a formal series with strict Weierstrass factorization. For every ring homomorphism $R\to R'$, the canonical map
	$(R[[t]]/(x))\otimes_R R' \to R'[[t]]/(x)$ of \eqref{tensor-completion} is an isomorphism. 
\end{corollary}

\begin{proof}
In the course of the proof of Prop.\ref{algebra}, we have showed that if $x=qu$ is a strict Weierstrass factorization of $x$, the $R$-linear map $R[t]/(q)\to R[[t]]/(x)$ is an isomorphism. The assertion can now be reduced to the obvious equality $R[t]/(q)\otimes_R R'=R'[t]/(q)$.
\end{proof}

For every prime ideal $\fp$ of $R$, we will denote $\nu_\fp:R\to k_\fp$ the canonical homomorphism from $R$ to the fraction field $k_\fp$ of $R/\fp$. A formal series $x\in R[[t]]$ is non-degenerate if and only if for every $\fp\in\Spec(R)$, $\nu_\fp(x)\in k_\fp[[t]]$ is nonzero. In this case, we will denote $\ord_\fp(x)$ the vanishing order of the formal series $\nu_\fp(x)\in k_\fp[[t]]$.

\begin{proposition} \label{criterion}
	Let $x\in R[[t]]$ be a non-degenerate formal series with coefficients in $R$. We consider the two assertions:
	\begin{enumerate}
\item[(1)] $x$ has a strict Weierstrass factorization;
\item[(2)] the function $\fp \to \ord_\fp(x)$ is a locally constant function on $\Spec(R)$.
	\end{enumerate}
For an arbitrary $k$-algebra $R$, (1) implies (2). The converse statement that (2) implies (1) holds under the assumption that ${\rm Nil}(R)^n=0$ for some $n\in \N$, ${\rm Nil}(R)$ being the nilradical of $R$. 
In particular if $R$ is either reduced or Noetherian then (1) and (2) are equivalent.
\end{proposition}

\begin{proof}
	Assume that $x$ has a strict Weierstrass factorization i.e. $x\in \cL^{\beta+} \bA^1(R)$, then by Cor. \ref{base-change}, $R[[t]]/(x)$ is a projective $R$-module of finite type compatible with base change. In particular, for every prime ideal $\fp$ we have 
	\begin{equation}\label{base-change-formula}
		k_\fp[[t]]/(x)= R[[t]]/(x) \otimes_R k_\fp
	\end{equation}
	where $\dim_{k_\fp}(k_\fp[[t]]/(x))=\ord_\fp(x)$. For $R[[t]]/(x)$ is a projective $R$-module of finite type, $\fp\mapsto \ord_\fp(x)$ is a locally constant function of $\fp$. 
	
	Now we assume that ${\rm Nil}(R)^n=0$ for some $n\in \N$. Let $x\in R[[t]]$ be a formal series with coefficients in $R$ such that the function $\fp \to \ord_\fp(x)$ is locally constant on $\Spec(R)$. We will prove that $x$ has a strict Weierstrass factorization. 
	We may assume that $\Spec(R)$ is connected and $\ord_\fp(x)=d$ for all $\fp\in\Spec(R)$. Let us consider the coefficients $x_0,x_1,\ldots$ of the formal series $x=x_0+x_1t+\cdots$. For $0\leq i\leq d-1$, we have $\nu_\fp(x_i)=0$ for all $\fp\in\Spec(R)$. It follows that $x_i$ belongs to ${\rm Nil}(R)$ the intersection of all prime ideals of $R$. We recall that ${\rm Nil}(R)$ can also be characterized as the ideal of all nilpotent elements of $R$. Let us denote $\bar R=R/{\rm Nil}(R)$, $\bar x_i$ the image of $x_i$ in $\bar R$, and $\bar x$ the image of $x$ in $\bar R[[t]]$. We have $\bar x_i=0$ for all $i$ in the range $0\leq i \leq d-1$. On the other hand, for every $\fp\in\Spec(R)$ we have $\nu_\fp(x_d) \neq 0$. Since $x_d$ doesn't lie in any maximal ideal of $R$, we have $x_d\in R^\times$. Now $\bar x=\bar x_d t^d + \bar x_{d+1} t^{d+1} +\cdots \in \bar R[[t]]$ with $\bar x_d\in \bar R^\times$ obviously has a strict Weierstrass factorization. Under the assumption ${\rm Nil}(R)^n=0$ for some $n\in N$, we deduce that $x$ has a strict Weierstrass factorization in virtue of Prop. \ref{strict-Weierstrass-Artin} that follows.\end{proof}

Prop, \ref{criterion} provides a handy way to check whether a formal series in $x\in R[[t]]$ has a strict Weierstrass factorization by calculating the function $\fp\mapsto \ord_\fp(x)$. For instant, for $R=k[a]$ the formal series $x=-a+t \in R[[t]]$ does not have strict Weierstrass factorization for $\ord_\fp(x)=0$ if $\fp\neq (a)$ and $\ord_\fp(x)=1$ if $\fp=(a)$.

We also observe that both statement and proof of the above proposistion can be simplified if $R$ is a $k$-algebra of finite type. For a $k$-algebras of finite type, we may replace in the statement of Prop, \ref{criterion} $\Spec(R)$ by the space $\Spm(R)$ of maximal ideals  For maximal ideals, the base change formula \eqref{base-change-formula} holds in virtue of Lemma \ref{base-change-lemma} so that we don't need to appeal to Prop. \ref{algebra}. This assures that (1) implies (2).  The converse statement holds because for a $k$-algebra of finite type $R$, the intersection of all maximal ideals, its Jacobson radical $\Jac(R)$, is equal to the intersection of all prime ideals, its nilradical. 

\begin{proposition}\label{strict-Weierstrass-Artin}
	The morphism $\cL^{\beta+}\bA^1 \to \cL^\bullet\bA^1$ is surjective and formally étale: for every $k$-algebra $R$, every nilpotent ideal $I$ of $R$ i.e. $I^n=0$ for some $n\in\NN$, and $\bar R=I/R$, if $x\in R[[t]]$ is a formal series whose reduction modulo $I$ satisfies $\bar x\in \bar R((t))^\times \cap \bar R[[t]]$ then $x\in R((t))^\times \cap R[[t]]$. 
	In particular, if $R$ is an Artinian local $k$-algebra then 
	$$R[[t]]\cap R((t))^\times= \cL^\bullet \bA^1(R).$$
	In other words, every non-degenerate formal series $x\in \cL^\bullet \bA^1(R)$ has a unique strict Weierstrass factorization under the assumption that $R$ is an Artinian local ring .
\end{proposition}

\begin{proof}
Let $I((t))$ denote the ideal of $R((t))$ consisting of all Laurent formal series whose  coefficients belong to $I$. We have an exact sequence
$$ 0\to I((t)) \to R((t)) \to \bar R((t)) \to 0 .$$
Since $I^n=0$, we have $I((t))^n=0$. It follows that the homomorphism $R((t)) \to \bar R((t))$ induces a bijection between their spectra. For an element in a ring is invertible if it doesn't belong to any maximal ideal, we deduce that $x\in R((t))$ is invertible if and only its reduction $\bar x\in \bar R((t))$ is invertible.
\end{proof}

\begin{proposition} \label{ind-scheme}
The functor $R\mapsto \cL^{\beta+}\bA^1(R)$ is representable by a strict ind-scheme. The projection on the degree of $q$ defines a bijection from the set of connected components of $\cL^{\beta+}\bA^1$ on $\N$. For every $d\in \N$, the component $\cL^{\beta+,d}\bA^1$ is isomorphic to $\hat Q_{d} \times \cL \Gm$ where $\hat Q_d$ is the completion of the space $Q_d$ of all monic polynomials of degree $d$ at the point $q=t^d$.
\end{proposition}

\begin{proof}
The uniqueness of the strict Weierstrass factorization implies that
	$$\cL^{\beta+,d}\bA^1(R) = Q^0_{d}(R) \times R[[t]]^\times$$
where $Q_d^0(R)$ is the set of monic polynomial of degree $d$ dividing $t^n$ for some $n\in N$. 

For each $n\in\N$, we denote $Q_d^{0,n}(R)$ the set of monic polynomials of degree $d$ dividing $t^n$. By definition $Q_d^{0,n}(R)$ is the subset of $Q_d(R)$ consisting of monic polynomials of degree $d$ such that there exists $q' \in Q_{n-d}(R)$ satisfying $qq'=t^{n}$. This is equivalent to say that $(t^{-d} q)(t^{-n+d} q')=1$ i.e. the elements $t^{-d} q, t^{-n+d} q'$ are inverse of each other in  $R[t^{-1}]^\times$. It follows that if $q'$ exists, it is unique. It follows that there is a natural bijection between $Q_d^{0,n}(R)$ and the subset of $Q_d(R)\times Q_{n-d}(R)$ of consisting of pairs $(q,q')$ such that $qq'=t^n$. It follows that the functor $R\mapsto Q_d^{0,n}(R)$ is representable by a $k$-scheme of finite type. Since $Q_d^{0,n}$ has a unique geometric point $(q,q')=(t^d,t^{n-d})$, it is a local Artinian scheme. 

We have a map $Q_d^{0,n}\to Q_d^{0,n+1}$ given by $(q,q')\mapsto (q,tq')$. Since the morphism $Q_{n-d}\to Q_{n-d+1}$ given by $q'\mapsto tq'$ is a closed immersion, the morphism $Q_d^{0,n}\to Q_d^{0,n+1}$ which derives from the latter by base change is also a closed immersion. 

We have thus proved that $Q^0_d$ is a strict colimit of Artinian schemes based at the point $t^d$ of $Q_d$. It follows that there is a canonical morphism of formal schemes $Q^0_d \to \hat Q_d$. It follows from Prop. \ref{strict-Weierstrass-Artin} that the morphism $Q^0_d \to \hat Q_d$ is an isomorphism.
\end{proof}

To summarize, the property of having a strict Weierstrass factorization is controlled by the morphism
\begin{equation} \label{gamma}
	\gamma:\cL^{\beta+} \bA^1 \to \cL^\bullet \bA^1
\end{equation} 
which induces isomorphism on $k$-points. We have seen that {\it $\gamma$ is formally étale}. On the negative side {\it $\gamma$ is not schematic} because otherwise $\cL^{\beta+} \bA^1$ would be schematic and it isn't according to its explicit description in Prop. \ref{ind-scheme}.

We will now study the geometry related to the notion of non-strict Weierstrass factorization. The natural action of $\Gm$ on $\bA^1$ induces on the level of arc spaces an action of $\cL \Gm$ on $\cL\bA^1$.  For every $d\in\N$ we have a morphism $\cL \bA^1 \to \cL_d \bA^1$ where $\cL_d\bA^1$ is the affine space $\Spec( k[x_0,x_1,\ldots,x_d])$. We will denote $\cL^{\leq d} \bA^1$ to be the open subscheme of $\cL\bA^1$ defined as the preimage of the complement of the point $0$ in $\bA^d$. The set of $k$-points of $\cL^{\leq d} \bA^1$ consists in the series $x\in k[[t]]$ such that $\val(x) \leq d$. We have the inclusions
$$\cL^{\leq 0}\bA^1 \subset \cL^{\leq 1}\bA^1\subset \cdots.$$
The union of those open subschemes is the space $\cL^\bullet \bA^1$ of non-degenerate arcs.

 We will consider the affine subspace $Q_d$ of $\cL\bA^1$ defined by the equations $x_d=1$ and $x_{d+1}=x_{d+2}=\cdots=0$. The action of $\cL\Gm$ on $\cL \bA^1$ induces a morphism
\begin{equation}\label{beta-d}
	\beta_d:Q_d \times \cL \Gm \to \cL^{\leq d} \bA^1
\end{equation}
given by $\beta(q,u)=uq$. By definition a formal series $x\in \cL \bA^1(R)$ has   Weierstrass factorization if it lies in the image $\beta_D(Q_d(R)\times \cL\Gm(R))$ for some $d\in\N$.

We will see that as opposed to the morphism $\gamma$ of \eqref{gamma}, the Weierstrass morphism $\beta_d$ of \eqref{beta-d} is not formally étale but only formally smooth in a weak sense. On the other hand, it has the advantage of being schematic. 

\begin{proposition}\label{beta-formal-smooth}
To simplify notation, we will write $X=Q_d\times \cL\Gm$ and $Y=\cL^{\leq d}\bA^1$.
	The morphism $\beta_d:X\to Y$ is surjective on $k$-points. Moreover, for every $x\in X(k)$  mapping on $y\in Y(k)$, the induced mapping on formal completions $\hat X_x\to \hat Y_y$ is formally smooth.	
\end{proposition}

\begin{proof}
The surjectivity on $k$-points is obvious. To prove that the induced mapping on formal completions $\hat X_x\to \hat Y_y$ is formally smooth is equivalent to prove
	the usual lifting properties for Artinian local algebras: for $R$ an Artinian algebra, $I$ a ideal of $R$, $\bar R=R/I$, for every $\bar x\in X(\bar R)$ mapping to $\bar y\in Y(\bar R)$, if $y\in Y(R)$ reducing to $\bar y$ modulo $I$ then there exists $x\in X(R)$ mapping to $y\in Y(R)$ and reducing to $\bar x\in X(\bar R)$ mod $I$. 

	Let $y\in Y(R)$ be a non degenerate formal series with coefficients in an Artinian local ring $R$. By Prop. \ref{strict-Weierstrass-Artin}, there exists a unique strict Weierstrass factorization $y=q_0u_0$ where $u_0\in R[[t]]^\times$ and $q_0$ is a monic polynomial of some degree $e$, $q_0$ dividing $t^n$ for some $n\in\NN$.
	
	The point $\bar x\in X(\bar R)$ mapping to $\bar y\in Y(\bar R)$ corresponds to a Weierstrass factorization $\bar y=\bar q \bar u$ of $\bar y$, its reduction  modulo $I$ where $\bar q$ is a monic polynomial of degree $d$ with coefficients in $\bar R$ and $\bar u\in \bar R[[t]]^\times$. 
	By reducing $y=u_0q_0$ modulo $I$ we have the unique strict Weierstrass factorization $\bar y=\bar u_0 \bar q_0$ of $\bar y$ that may be different from the Weierstrass factorization $\bar y=\bar u \bar q$ given by $\bar x\in X(\bar R)$.
	
	Since $t^n\in (\bar q_0)=(\bar x)=(\bar q)$, we have 
	$$\bar R[[t]]/\bar q=\bar R[t]/(t^n,\bar q)$$
i.e. $R[[t]]/\bar q$ is a quotient of $\bar R[t]/\bar q$. On the other hand, the strictness property shows that $\bar R[[t]]/(\bar q_0)=\bar R[t]/(\bar q_0)$ as in the proof of Prop. \ref{algebra}. It follows there exists a surjective $R$-linear homomorphism of free $R$-modules $\bar R[t]/\bar q \to \bar R[t]/\bar q_0$ compatible with action of $t$. In particular we have the inequality $d\geq e$ by comparing their ranks. Moreover by taking characteristic polynomials of $t$ acting on those free modules, we see that $\bar q_0$ divides $\bar q$ i.e. there exists a monic $\bar p$ of degree $d-e$ such that $\bar q=\bar q_0 \bar p$. We have $\bar p\in Q_d(\bar R)\cap \bar R[[t]]^\times$.

Now there are no problem to lift the polynomial $\bar p\in Q_{d-e}(\bar R)$ to a monic polynomial $p\in Q_{d-e}(R)$ as $Q_{d-e}$ is an affine space. Choose an arbitrary lift $p\in Q_{d-e}(R)$ of $\bar p\in Q_{d-e}(\bar R)$. Since the reductio mod $I$ of $p$ is an invertible element $\bar p\in \bar R[[t]]^\times$, we have $p\in R[[t]]^\times$. Now we define the lift $y\in Y(R)$ of $\bar y=(\bar q,\bar u)$ to be $y=(q_0 p, p^{-1} u_0)$.
\end{proof}

\begin{proposition} \label{beta-prop} Let $\cL_n\Gm$ and $\cL_n\bA^1$ denote the spaces of $n$-jets on $\Gm$ and $\bA^1$ for every $n\in\NN$.
If $d\leq n\in\NN$, then the morphism
\begin{equation} \label{beta-dn}
	\beta_{d,n}:Q_d \times \cL_n \Gm \to  \cL_n^{\leq d} \bA^1
\end{equation}
given by $(q,u) \mapsto qu$ is smooth. 
\end{proposition}

\begin{proof} If we denote $X=Q_d \times \cL \Gm$, $X_n=Q_d \times \cL_n \Gm$, $Y=\cL_n^{\leq d} \bA^1$ and $Y_n=\cL_n^{\leq d} \bA^1$ then we have a commutative diagram
\begin{equation}
	\begin{tikzcd}
X \arrow{r}{\beta_d} \arrow{d}[swap]{\mu_n}
& Y \arrow{d}{\nu_n} \\
X_n \arrow{r}[swap]{\beta_{d,n}}
& Y_n
\end{tikzcd}
\end{equation}
where $\mu_n$ and $\nu_n$ are projective limits of systems of smooth surjective maps with smooth surjective transition morphisms. For every $x_n\in X_n(k)$ mapping to $y_n\in Y_n(k)$, there exists $x\in X(k)$ over $x_n$ mapping to $y\in Y(k)$ over $y_n$. We have a commutative diagram of formal completions
\begin{equation}
	\begin{tikzcd}
\hat X_x \arrow{r}{\hat\beta_d} \arrow{d}[swap]{\hat\mu_n}
& \hat Y_y \arrow{d}{\hat\nu_n} \\
\hat X_{n,x_n} \arrow{r}[swap]{\hat\beta_{d,n}}
& \hat Y_{n,y_n}
\end{tikzcd}
\end{equation}
where $\hat\beta_d,\hat\mu_n,\hat\nu_n$ are known to be formally smooth. It follows that $\hat\beta_{d,n}$ is also formally smooth. Since $X_n$ and $Y_n$ are $k$-schemes of finite type, we can derive from the smoothness of $\beta_{d,n}$ from the formal smoothness of induced morphism on formal completions of points.
\end{proof}

Here we prove that a certain morphism between finite order jet spaces, which are of finite type, is smooth by proving that the morphism on the level of arc spaces is formally smooth in a weak sense. It is also possible to prove directly the smoothness of the morphism between jet spaces by calculation on their tangent spaces as we will see later in the study of the Weierstrass division theorem, see the proof of Prop. \ref{alpha-smoothness}.

Let us draw two consequences of the smoothness of $\beta_{d,n}$. 

\begin{corollary}
	Let $R$ be a local Henselian ring of maximal ideal $\fm$ and residue field $k$. Let $x\in R[[t]]$ of nonzero reduction $\bar x\ in\ k[[t]]$ mod $\fm$. Let $\val(\bar x)=d$. Then for every $n > d$ there exists a monic polynomial $q\in Q_d(R)$ with coefficients in $R$ and $u\in R[[t]]^\times$ such that 
	$$qu \equiv x \mod t^n.$$
\end{corollary}

\begin{proof}
	This is the Henselian lifting property for smooth morphisms.  
\end{proof}

\begin{corollary}
	The morphism $\beta_d:Q_d\times \cL\Gm \to \cL^{\leq d}\bA^1$ is a projective limit of a system of smooth morphisms with possibly non-smooth transition morphisms.
\end{corollary}

\begin{proof}
By base change we obtain smooth morphisms
$$\beta_{d,n}^\infty:(Q_d \times \cL_n \Gm) \times_{\cL\bA^1} \cL_n\bA^1 \to \cL\bA^1$$
that form a projective system whose limit if $\beta_d$. One should pay a special attention to the fact that the transition maps in this system are not smooth. Indeed, on finite levels, we have a commutative diagram
\begin{equation}
\begin{tikzcd}
Q_d \times \cL_m \Gm \arrow{r}{\beta_m} \arrow{d}[swap]{p_1}
& \cL_m \bA^1 \arrow{d}{p_2}\\
Q_d \times \cL_n \Gm \arrow{r}[swap]{\beta_n}
& \cL_n \bA^1
\end{tikzcd}
\end{equation}
for all integers $m\geq n$ in which all maps are smooth. The diagram is however not Cartesian and it induces a morphism
\begin{equation}
(Q_d \times \cL_m \Gm) \to (Q_d \times \cL_n \Gm) \times_{\cL_n \bA^1}\cL_m \bA^1 	
\end{equation} 
that is not smooth. Thus the transition morphisms in the projective system $\plim_{n}\beta_{d,n}^\infty=\beta_d$ are not smooth.
\end{proof}

Based on the discussion on the property of the Weierstrass morphism $\beta$ of \eqref{beta-d}, we will introduce a new type of morphisms which seems to be useful in dealing with arc spaces.

\begin{definition}\label{Weierstrass-type}
	A morphism $f:X\to Y$ of affine $k$-schemes is said to be of Weierstrass type if it satisfies the following properties:
	\begin{itemize}
		\item $f$ is a filtered limit of a projective system of smooth morphisms in which transition maps are not required to be smooth;
		\item for every $k$-point $x\in X(k)$ with image $y\in Y(k)$, the induced morphism on formal completions $\hat X_x \to \hat Y_y$ is formally smooth.
	\end{itemize}
	If the codomain is the point $Y=\Spec(k)$, we will say that $X$ is a $k$-scheme of Weierstrass type.
\end{definition}

 Arbitrary base changes of a morphism of Weierstrass type are of Weierstrass type. 
Composition of morphisms of Weierstrass type is also of the same type.
As examples, the Weierstrass maps $\alpha$ and $\beta$  of \eqref{beta-d} is of Weierstrass type.

We will introduce a much stronger notion:

\begin{definition}
\begin{itemize}
	\item A morphism $f:X\to Y$ is said to be essentially smooth surjective if it is the filtered limit of a projective system of smooth surjective morphism in which transition maps are also smooth surjective.
	\item A $k$-scheme $X$ is said to be essentially smooth if the structural morphism  $f:X\to \Spec(k)$ is essentially smooth surjective.
	\item A $k$-scheme $X$ is said to be essentially of finite type if there exists an essentially smooth surjective morphism $f:X\to S$ where $S$ is a $k$-scheme of finite type. 
\end{itemize}
\end{definition}

A $k$-scheme essentially of finite type is of course also of Weierstrass type. The converse is wrong. There are much more $k$-schemes of Weierstrass type than $k$-schemes essentially of finite type. We will prove an useful criterion to recognize a morphism of Weierstrass type.

\begin{proposition}\label{criterion-Weierstrass}
	Let $X$ and $Y$ be $k$-schemes essentially of finite type. If $f:X\to Y$ be a morphism that induces formally smooth morphisms on formal completions of $k$-points, then $f:X\to Y$ is a morphism of Weierstrass type. 
\end{proposition}

\begin{proof}
Let us write $X=\Spec(A)$ where $A=\ilim (A_i)$ with $A_i$ being $k$-algebras of finite type and $X_i=\Spec(A_i)$. Similarly we write $Y=\Spec(B)$ where $B=\ilim (B_j)$ with $B_j$ being $k$-algebras of finite type and $Y_j=\Spec(B_j)$. The restriction of $\phi:B\to A$ to $B_j$ has image of finite type so that $\phi(B_j) \subset A_i$ for some $i$. We have a commutative diagram
\begin{equation}
	\begin{tikzcd}
X \arrow{r}{f} \arrow{d}[swap]{\mu_i}
& Y \arrow{d}{\nu_j} \\
X_i \arrow{r}[swap]{f_{i,j}}
& Y_j
\end{tikzcd}
\end{equation}
Let $x\in X(k)$, $y=f(x)\in Y(k)$. If we denote $x_i$ and $y_j$ the images of $x$ and $y$ in $X_i$ and $Y_j$ respectively then we have $f_{i,j}(x_i)=y_j$. We consider the diagram of formal completions
\begin{equation}
	\begin{tikzcd}
\hat X_x \arrow{r}{\hat f_x} \arrow{d}[swap]{\hat \mu_{i,x}}
& \hat Y_y \arrow{d}{\hat \nu_{j,y}} \\
\hat X_{i,x_i} \arrow{r}[swap]{\hat f_{i,j,x_i}}
& \hat Y_{j,y_j}
\end{tikzcd}
\end{equation}
where $\hat f_{x}, \hat \mu_{i,x}, \hat \nu_{j,y}$ are formally smooth. It follows that $\hat f_{i,j,x_i}$ is also formally smooth.
For $X_i$ and $Y_j$ are $k$-schemes of finite type, the morphism $f_{i,j}:X_i\to Y_j$ is smooth at $x_i$. Since $X\to X_i$ is surjective, this implies that $f_{i,j}:X_i\to Y_j$ is smooth. 
Now $f:X\to Y$ can be realized as the projective limit of smooth morphisms $X_i\times_{Y_j} Y \to Y$ as $j\to \infty$. 
\end{proof}

\section{Variations on the Weierstrass division theorem}

In the classical setting, the Weierstrass division theorem asserts that if $f$ is a germ of holomorphic functions at the origin of $\C^{n+1}$ given with coordinates $x_1,\ldots,x_n,t$, $q$ is a polynomial of variable $t$ of degree $d$ of the form $q=q_0 + q_1 t +\cdots q_{d-1} t^{d-1} + t^d$ where $q_0,\ldots,q_{d-1}$ are germs of holomorphic functions at the origin of $\C^n$ equipped with coordinates $x_1,\ldots,x_n$ and vanishing at the origin, then there exist an unique expression $f=qh+a$ where $h$ is a germ of holomorphic function at the origin of $\C^{n+1}$ and $a$ is a polynomial of variable $t$ of degree less than $d$ whose coefficients are germs of holomorphic functions at the origin of $\C^n$. 

In this section we aim to study the Weierstrass division problem when we replace the ring of germs of holomorphic functions at the origin of $\C^n$ by an arbitrary $k$-algebra $R$, $f$ by a formal series with coefficients in $R$, and $q\in R[t]$ a monic polynomial of degree $d$.  

Recall that we have denoted $Q_d$ the $d$-dimensional affine space representing the functor $R\mapsto Q_d(R)$ of all monic polynomials of degree $d$ with coefficients in $R$.
Let us now denote $A_d$ the $d$-dimensional affine space representing the functor $R\mapsto A_d(R)$ of all polynomials of degree less than $d$. The Weierstrass division problem for a varying test ring $R$ consists in studying the morphism
\begin{equation}\label{alpha-d}
	\alpha_d:Q_d\times A_d \times \cL \bA^1 \to Q_d \times \cL \bA^1
\end{equation}
given by the formula:
\begin{equation}
	\alpha_d(q,a,v)=(q,qv+a).
\end{equation}
In a sense, the morphism $\alpha_d$ is the derivative of the morphism $\beta_d$ of \eqref{beta-d}.

The morphism $\beta_d:X\to Y$ of \eqref{beta-d} where $X=Q_d\times \cL \Gm$ and $Y=\cL^\bullet \bA^1$ induces a morphism of tangent bundles
\begin{equation} \label{dbeta}
	\d \beta_d:\rT X \to \rT Y\times_Y X.
\end{equation} 
where $\rT X= Q_d\times \cL \Gm \times A_d \times \cL \bA^1$, $\rT Y=\cL^\bullet \bA^1 \times \cL \bA^1$, and $\rT Y\times_Y X= Q_d\times \cL \Gm \times \cL \bA^1$. If $(q,u;a,v)\in Q_d\times \cL \Gm \times A_d \times \cL \bA^1$ then 
\begin{equation}
	\d \beta_d(q,u,a,v)=(q,u,ua+qv)
\end{equation}
which is a point of $Q_d\times \cL \Gm \times \cL \bA^1$. We obtain $\alpha_d$ from $\d\beta_d$ by setting the parameter $u$ to be $1$.

\begin{proposition} \label{alpha-smoothness}
	For every integer $d\in\NN$, the morphisms $d\beta_d$ of \eqref{dbeta} and  $\alpha_d$ of \eqref{alpha-d} are surjective and induces formally smooth morphisms on formal completions of points. Over the point $q=t^d$, it induces isomorphism on formal completions of points.
\end{proposition}

\begin{proof}
Because $\beta_d$ is a morphism of Weierstrass type, so is its derivative $\d\beta_d$. By base change, $\alpha_d$ is also a morphism of Weierstrass type. The same argument applies for the property of being surjective on $k$-points.
Over the point $q=t^d$, $\beta_d$ induces an isomorphism on formal completions of points, so are $\d\beta_d$ and $\alpha_d$. 

One could also reverse the argument and prove first that $\alpha_d$ induces a smooth morphism at finite levels. Indeed, for every $n\geq d$, the morphism
\begin{equation}\label{alpha-d-n}
	\alpha_{d,n}: Q_d\times A_d \times \cL_n \bA^1 \to Q_d \times \cL_n \bA^1
\end{equation} 
is a linear morphism of finite rank vector bundles over $A_d$. It is enough to prove that $\alpha_{d,n}$ induces a surjective linear morphism in every fiber. It is enough to prove that for a monic polynomial $q\in Q_d(k)$ of degree $d$ with coefficients in $k$, the $k$-linear map
$$A_d(k) \times k[[t]]/(t^n) \to k[[t]]/(t^n)$$
given by the formula \eqref{alpha-d} is surjective. 
We will prove that the $k$-linear map
$$A_d(k) \times k[[t]] \to k[[t]]$$
given by the formula \eqref{alpha-d} is surjective. This is equivalent to saying that the map 
$$A_d(k) \to k[[t]]/(q)$$ 
is surjective. This map factors as
$$A_d(k) \to k[t]/(q) \to k[[t]]/(q)$$
where $A_d(k) \to k[t]/(q)$ is an isomorphism by the Euclidean division theorem. The map $k[t]/(q) \to k[[t]]/(q)$ is surjective as $k[[t]]/(q)=k[t]/(t^e)$ where $e$ is the multiplicity of $t$ dividing $q$.
\end{proof}

We will now exhibit a strange family of finite-dimensional schemes of Weierstrass type. For every $d\in\N$, we denote $S_d$ the kernel of morphism $\alpha_d$ of \eqref{alpha-d}, in other words, we have the Cartesian diagram
\begin{equation} \label{S-d}
	\begin{tikzcd}
S_d \arrow{r}{} \arrow{d}[swap]{\alpha_K}
& Q_d \times A_d \times \cL \bA^1 \arrow{d}{\alpha_d} \\
Q_d \arrow{r}[swap]{0}
& Q_d \times \cL\bA^1
\end{tikzcd}
\end{equation}
where $0$ is the zero section $q\mapsto (q,0)\in A_d\times \cL\bA^1$.
The scheme $S_d$ is a generalized vector bundle over $Q_d$ in the sense that its fiber over each point $q\in Q_d(k)$ is a $k$-vector space.
The dimension of the fibers varies however in the opposite direction of what happens with a morphism of finite type. Indeed, for every monic polynomial $q\in Q_d(k)$ of degree $d$, the fiber $S_{d,q}(k)$ of $S_d$ over $q$ is a $k$-vector space of dimension $d-e$ where $e$ is order of vanishing of $q$ at $0$. In particular, if $q(0)\neq 0$ then $\dim  S_{d,q}(k)=d$, and if $q=t^d$ then $S_{d,q}(k)=0$.

Let us compute explicitly the ring of coordinates of $S_1$. We will write $Q_1=\Spec(k[q_0])$ with the universal monic polynomial of degree one $q=q_0+t \in Q_1(k[q_0])$. We will denote $A_1=\Spec(k[a_0])$ the space of polynomial of degree less than one and $\cL \bA^1=k[x_0,x_1,\ldots]$. The scheme $S_1$ is defined as the closed subscheme of $Q_1\times A_1\times \cL \bA^1$ defined by the system of equations 
\begin{eqnarray*}
	a_0 + q_0 x_0 &=& 0 \\
	x_0 + q_0 x_1 &=& 0 \\
	x_1 + q_0 x_2 &=& 0\\
	\ldots
\end{eqnarray*}
In other words, we have $S_1=\Spec(R_1)$ where
\begin{equation} \label{equations-S1}
 	R_1=k[q_0,a_0,x_0,x_1,\ldots,]
/(a_0+q_0x_0,x_0 +q_0 x_1,\ldots). 
\end{equation}
The fiber of $S_1$ over a point with coordinate $q_0\neq 0$ in $Q_1$ is the affine line since the coordinate $a_0$ will determine other coordinates by $x_0=-{a}^{-1} a_0$, $x_1=-{a}^{-1}x_0$ ... On the other hand, the fiber of $S_1$ over the point $q_0=0$ in $Q_1$ is reduced to the point of coordinates $a_0=0$ and $x_0=x_1=\cdots =0$.

We may look at $S_1$ from three different angles. First, from the equations defining $S_1$ we see that $S_1$ is a projective limit of affine planes, in other words its ring of coordinates $R_1$ is an inductive limit $R_1=\ilim A_n$ with $A_n=k[q_0,x_n]$ with transition maps $A_n \mapsto A_{n+1}$ given by $x_n \mapsto -q_0 x_{n+1}$. The morphism $\Spec(A_{n+1}) \to \Spec(A_n)$ can be realized as the blow-up of the affine plane $\bA^2=\Spec(A_n)$ at the origin $(q_0,x_n)=(0,0)$ with the strict transform of the $x_n$-axis deleted.

Second, set theoretically speaking, $S_1$ looks like a "sawed plane" i.e. the plane of coordinates $(q,x)$ with the $x$-axis $(0,x)$ deleted and the origin $(0,0)$ added back in. This constructible subset of the plane can be thus realized as the set of points of a scheme which is of course not of finite type. 

Third, the sawed plane $S_1$ may be seen in some sense as the glueing of the plane $(q,x)$ with the $x$-axis deleted and with the $q$-axis $(q,0)$. Indeed, if $A$ is a $k$-algebra of finite type and $\phi:R_1\to A$ be a $A$-point of $S_1$, mapping some prime ideal $\fp_A$ of $A$ to the prime ideal of $R_1$ generated by $\fp_0=(q_0,a_0,x_0,\ldots)$ corresponding to the origin $(0,0)$ of $S_1$, then the morphism $\Spec(A)\to S_1$ must factor through the $q$-axis. Indeed, under this assumption we have $\phi(q_0)\in \fp_A$. Since the variables $a_0,x_0,\ldots$ are divisible by every power of $q_0$, we have $\phi(a_0),\phi(x_0) , \ldots \in \fp_A^n$ for every $n\in\N$. Since $A$ is a $k$-algebra of finite type, it follows that $\phi(a_0)=\phi(x_0)=\cdots=0$.

Let us identify the $q$-axis with the subscheme of $S_1$ defined by the equations $a_0=x_0=x_1=\cdots=0$. We can show that that the completion of $S_1$ at origin is isomorphic to the completion of the $q$-axis at the origin defined by the maximal ideal $\fq$ of $R_1$. Indeed, for every $n$, we have $a_0,x_0,x_1,\ldots \in \fq^n$ and therefore $R_1/\fp_0^n = k[q_0]/q_0^n$.

\begin{proposition} For every $d\in\N$, let $S_d$ be the kernel of the homomorphism $\alpha_d$ of \eqref{alpha-d} defined by the Cartesian diagram \eqref{S-d}.
	\begin{enumerate}
		\item $S_d$ is the projective limit of kernels of homomorphisms $\alpha_{d,n}$ of \eqref{alpha-d-n} which are vector bundles of rank $d$ over $Q_d$ as $n\to\infty$. In particular $S_d$ is a scheme of Weierstrass type.
		\item As topological space $S_d$ is the constructible subspace of $Q_d\times A_d$ whose fiber over $q\in Q_d(k)$ is a $k$-vector space of rank $d-e$ where $e$ is the order of vanishing of $q$ at $t=0$. In particular, $S_d$ contains the zero section $Q_d\times \{0\}$ of $Q_d\times A_d$. 
		\item For every $k$-algebra of finite type $A$ and $\phi:\Spec(A)\to S_d$ a $A$-point of $S_d$ whose image contains the point $(t^d,0)\in Q_d\times A_d$. Then $x$ factors through the zero section $Q_d\times \{0\}$.
		\item The completion of $S_d$ at the point $(t^d,0)$ is isomorphic to the completion of the zero section at the same point. 
	\end{enumerate}
\end{proposition}

\begin{proof}
	The proof proceeds in the same way as the calculations in the case $d=1$ that we have just outlined.
\end{proof}

We will now consider another family of strange schemes springing out of Weierstrass division problem. For every $d\in\N$, we consider the homomorphism
\begin{equation}
	\mu_d: Q_d\times \cL \bA^1 \to Q_d\times \cL \bA^1
\end{equation}
defined by $(q,x)\mapsto (q,qx)$, and define $Z_d$ to be the kernel of $\mu_d$. For every $k$-algebra $R$, the set $Z_d(R)$ consists of pairs $(q,x)$ where $q\in R[t]$ is a monic polynomial of degree $d$ and $x\in R[[t]]$ a formal series with coefficients in $R$ satisfying $qx=0$. 

We may also define $Z_d$ as the closed subscheme of $S_d$ consisting of triples $(q,a,x)\in Q_d\times A_d\times \cL \bA^1$ with $a=0$. Set theoretically, there is no difference between $Z_d$ and the zero section $Q_d\times \{0\}$ as for every $q\in Q_d(k)$, the multiplication by $q$ is injective in $k[[t]]$. These spaces are however different as schemes, in other words $Z_d$ is a infinitesimal thickening of $Q_d\times \{0\}$. 

Let us consider again the case $d=1$ and describe explicitly the coordinate ring of $Z_1$.  
Using the equations \eqref{equations-S1} of $S_1$, we have
$Z_1=\Spec(V_1)$ where
\begin{equation} \label{equations-Z1}
 	V_1=k[q_0,x_0,x_1,\ldots,]
/(q_0x_0,x_0 +q_0 x_1,x_1+qx_2,\ldots). 
\end{equation}
The algebra $V_1$ can be realized as the inductive limit of $V_{1,n}=k[q_0,x_n]/(q_0^{n+1} x_n)$ with respect to transition maps $V_{1,n}\to V_{1,n+1}$ given by $x_n \mapsto -q_0 x_{n+1}$. We note that the transition maps $V_{1,n}\to V_{1,n+1}$ are all injective so that $x_n\neq 0$ in $V_1$. The scheme $Z_1$ is therefore the projective limit of schemes $Z_{1,n}=\Spec(V_{1,n})$ which topologically is a union of the $q_0$-axis and the $x_n$-axis as $n\to\infty$. We see that in the limit, the $x_n$-axis all contract topologically to the origin although they don't algebraically in the sense that $x_n\neq 0$. The ideal of $V_1$ generated by $x_0,x_1,\ldots$ defined the closed subscheme $Q_1\times \{0\}$ of $Z_1$.

\begin{proposition} \label{infinite-radical-V-d}
	 Let $Z_d=\Spec (V_d)$ and $I_d$ the ideal of $V_d$ defining the zero section $Q_d\times \{0\}$. Then we have $I_d^2=0$. Moreover every element $f\in I_d$ vanishes at every point $z\in Z_d$ at infinite order.  
\end{proposition}

\begin{proof}
	Let us restrict ourselves to the case $d=1$ as the general case is completely similar.
We first prove that $x_m x_n=0$ for all $m,n\in \N$. Since $x_m=(-1)^n q_0^{n+1} x_{m+n+1}$, we have $x_m x_n=(-1)^{n+1}  x_{m+n+1} q_0^{n+1} x_n=0$. 

We next prove that $x_n$ vanishes to an infinite order at every point of $\Spec(V_1)=\Spec (k[q_0])$. At a point $z$ of coordinate $q_0\neq 0$, the equation $q_0^{n+1} x_n=0$ implies that $x_n=0$ in the local ring $V_{1,z}$ in which $q_0$ is invertible. At the point $z$ of coordinate $q_0=0$, as $q_0$ belongs to the maximal ideal $\fm_z$, $x_n$ belongs to every power of $\fm_z$.
\end{proof}

\begin{definition} \label{nil-infinity}
For every commutative ring $R$, we will denote 
\begin{equation}
	\nil_\infty(R)=\bigcap_{\fp\in\Spec(R)} \bigcap_{n\in\N} \fp^n,
\end{equation}
the infinite nilradical of $R$, the intersection of all powers of all prime ideals of $R$. 
\end{definition}

Let us first collect some easy facts about the infinite nilradical. First, the infinite nilradical is obviously contained in the nilradical
$$\nil(R)=\bigcap_{\fp\in\Spec(R)} \fp,$$
which consists of all nilpotent elements of $R$. In particular, if $R$ is reduced then both $\nil(R)$ and $\nil_\infty(R)=0$. Second, as we can determine whether an element $x\in R$ belongs to the nilradical by evaluating it via the homomorphisms $\nu_\fp:R\to R/\fp$ for all prime ideal $\fp\in \Spec(R)$, we can also determine whether an element $x\in R$ belongs to the infinite nilradical by evaluating it via homomorphisms $\nu_\fp:R\to R/\fp^n$ for all $\fp\in\Spec(R)$ and $n\in \N$. In particular, $\nil_\infty(R)=0$ if and only if for every non-zero element $x\in R$, there exists a ring homomorphism $\nu:R\to A$ from $R$ to an Artinian local ring such that $\nu(x)\neq 0$.

\begin{proposition}
For a Noetherean ring $R$, we have $\nil_\infty(R)=0$.
\end{proposition}

\begin{proof}
Let $\{\fp_1,\ldots,\fp_r\}$ denote the set of associated primes of $R$.
By the primary decomposition theorem \cite[Thm. 6.8]{Matsumura:1986wj}, there exist primary ideals $I_1,\ldots,I_r$ with $\sqrt{I_i}=\fp_i$ such that $I_1\cap\cdots\cap I_r=0$.
As $\sqrt{I_i}=\fp_i$, for every $i$, there exists $m_i$ such that $\fp_i^{m_i} \subset I_i$. It follows that $\nil_\infty(R)\subset  I_1\cap\cdots\cap I_r$ and therefore $\nil_\infty(R)=0$. 
\end{proof}

Here is an example of commutative ring $R$ with $\nil_\infty(R)\neq 0$ deriving from our study of Weierstrass division theorem.
By Prop. \ref{infinite-radical-V-d}, $\nil_\infty(V_d)$ is a nonzero ideal generated by $x_0,x_1,\ldots$. Here is another example familiar in the $p$-adic Hodge theory. Let $\cO$ denote the ring of integers of $\C_p$ the $p$-adic completion of the algebraic closure of the field $\Q_p$ of $p$-adic numbers. Then $\cO$ is a local ring with maximal ideal $\fm$ satisfying $\fm=\fm^2$. In particular $\fm \subset \nil_\infty(\cO)$ and the latter is nonzero.

\begin{definition}
	A morphism of $k$-schemes $f:X\to Y$ is said to be a $\nil_\infty$-quasi-isomorphism if it induces a bijection $X(R)\to Y(R)$ on every $k$-algebra $R$ with $\nil_\infty(R)=0$.
\end{definition}

For instance, for every commutative ring, the closed embedding $\Spec(R/\nil_\infty(R))\to \Spec(R)$ is a $\nil_\infty$-quasi-isomorphism. In particular, the morphism $Q_d \to Z_d$ in Prop. \ref{infinite-radical-V-d} is a $\nil_\infty$-quasi-isomorphism. Here is a related and very instructive example of $\nil_\infty$-quasi-isomorphism which is not an isomorphism.

	Let $X=\Spec(k[x,y]/(xy))$ denote the union of the $x$-axis and $y$-axis in the plane of coordinates $(x,y)$. Let $X'=X-\{0\}=\Spec (k[x^{\pm 1}]) \sqcup \Spec (k[y^{\pm 1}])$ denote the complement of the origin in $X$. We consider the space $\cL^\bullet X$ of arcs on $X$ which are non-degenerate with respect to the open subset $X'$. The inclusion of the $x$-axis in $X$ induces a morphism $\cL^\bullet \bA^1 \to \cL^\bullet X$ and so does the $y$-asis. We have thus a morphism
	\begin{equation} \label{arc-on-x-axis-or-y-axis}
		\cL^\bullet \bA^1 \sqcup \cL^\bullet \bA^1 \to \cL^\bullet X
	\end{equation} 
	which induces a bijection on $k$-points. In fact in induces a bijection on $R$-points for all $k$-algebras $R$ with $\nil_\infty(R)=0$  by \ref{injective} but not in general by \ref{infinite-radical-V-d}.

\begin{lemma}\label{injective}
	Let $R$ be a commutative ring with $\nil_\infty(R)=0$, $x\in R[[t]]$ a non-degenerate formal series. Then the multiplication by $x$ in $R[[t]]$ is injective. 
\end{lemma}

\begin{proof}
	First we assume that $R$ is an Artinian local ring of maximal ideal $\fm$ and residue field $k_R$. In this case, the non-degeneracy condition means $\bar x\in k_R[[t]]$ is non zero in other words $\bar x\in k_R[[t]] \cap k_R((t))^\times$. It follows that $x\in R[[t]] \cap R((t))^\times$. Because the map $R[[t]] \to R((t))$ is injective, and the multiplication by $x$ is a bijection on $R((t))$, the multiplication by $x$ is injective on $R[[t]]$.
	
	Let $R$ be a ring with zero infinite radical and $x\in R[[t]]$ a non-degenerate formal series. Let $y\in R[[t]]$ be a formal series such that $xy=0$. Then for every ring homomorphism $\nu:R\to A$ from $R$ to Artinian local ring, we have $\nu(y)=0$. If $\nil_\infty(R)$ this implies $y=0$.
\end{proof}

We will use the Newton method to analyze the arc space of certain complete intersection following Drinfeld's paper \cite{Drinfeld:2002vda}. The difference with \cite{Drinfeld:2002vda} is that we will have to work over all $k$-algebras $R$ while Drinfeld restricts himself to Artinian rings. We will son realize that Drinfeld's arguments generalize well only under the assumption $\nil_\infty(R)=0$ for some familiar linear algebra facts only hold over rings with zero infinite radical.

\begin{lemma} \label{uniqueness}
Let $R$ be a commutative ring with $\nil_\infty(R)=0$. Let $\phi \in \Mat_n(R[[t]])$ a $n\times n$-matrix with coefficients in $R[[t]]$ whose determinant $\xi=\det(\phi)\in R[[t]]$ in a non-degenerate formal series as defined in \ref{non-degenerate-formal-series}. Then the associated linear operator  
$$\phi: R[[t]]^n \to R[[t]]^n$$
is injective, and its image consists of vectors $v\in R[[t]]^n$ such that
$$\phi' v \equiv 0 \mod \xi$$
where $\phi'$ is the adjugate matrix of $\phi$.
\end{lemma}

\begin{proof}
	By the Cramer formula we have $\phi'\phi=\phi\phi'=\xi \id_n$
where $\id_n$ is the identity matrix. By Lem. \ref{injective} the multiplication by $\xi$ in $R[[t]]^n$ is injective. It follows that both linear operators $\phi$ and $\phi'$ are injective.  

For $u,v\in R[[t]]^n$, the injectivity of $\phi,\phi'$ and the Cramer rule together imply that $\phi(v)=u$ if and only if $\xi v=\phi'(u)$. In other words $u\in \im(\phi)$ if and only if $\phi'(u) \equiv 0 \mod \xi$.
\end{proof}

\section{Variations on the Newton method}

As in \cite{Drinfeld:2002vda}, we will assume that $X$ is defined as the central fiber of a morphism $p:\bA^m\to \bA^n$ with $m\geq n$. The coordinates of $f=(f_1,\ldots,f_n)$ are regular functions on $\bA^m$. We will use the coordinates $x=(x_1,\ldots,x_{m-n},y_1,\ldots,y_n)$ for $x\in\bA^n$. We will consider the $n\times n$-matrix
\begin{equation} \label{phi}
	B(x)= \left( {\partial f_i \over \partial y_j} \right)^{i=1,\ldots,n}_{j=1,\ldots,n}
\end{equation}
whose coefficients are regular functions on $\bA^n$. We will assume that
\begin{equation}
	\xi(x)=\det(B(x))
\end{equation}
is a non-zero regular function on $\bA^m$. Let $Z$ denote the closed subscheme of $\bA^m$ defined by the equation $\xi=0$ and $U$ its complement open subscheme. We will denote 
$$\cL^\bullet \bA^m=\cL \bA^m - \cL Z$$ 
the non-degenerate arc space of $\bA^m$ with respect to the open subset $U$.

The Taylor expansion of $f$ with respect to the variables $y_1,\ldots,y_n$ is of the form
\begin{equation} \label{Taylor}
	f(x+v)=f(x)+ B(x;v) + B_2(x;v)+B_3(x;v)+... 
\end{equation}
for $x=(x_1,\ldots,x_{m-n},y_1,\ldots,y_n)\in \bA^m$ and $v=(v_1,\ldots,v_n)\in \bA^n$ and 
\begin{equation}\label{addv}
x+v=(x_1,\ldots,x_{m-n},y_1+v_1,\ldots,y_n+v_n).
\end{equation}
Here $B(x;v)$ is a function linear with respect to the variable $v$ given by the matrix $B(x)$, and more generally $B_n(x;v)$ is a homogenous polynomial of degree $n$ of variables $v$ whose coefficients are polynomials of $x$.

We consider the Cartesian diagram
\begin{equation}\label{N-0}
\begin{tikzcd}
N_0 \arrow{r}{} \arrow{d}[swap]{\alpha_0}
&  \cL^\bullet \bA^m
 \times \cL \bA^n \arrow{d}{\alpha} \\
\cL^\bullet X \arrow{r}[swap]{}
& \cL \bA^m
\end{tikzcd}
\end{equation}
with $\alpha(x,v_0)=x+ t\xi(x)v_0$ where the addition is defined as in \eqref{addv}. For $\alpha$ is a surjective morphism of Weierstrass type, so is $\alpha_0$.
For every $k$-algebra $R$, $N_0(R)$ is the set of $(x,v_0)\in R[[t]]^m \times R[[t]]^n$ such that 
\begin{equation} \label{equation-N1}
	f(x+t\xi(x)v_0)=0.
\end{equation}
and $\xi(x)$ is a non-degenerate formal series.

We will replace $N_0$ by another infinite dimensional scheme defined by congruence equations. 
Let $N_1$ be the affine $k$-scheme that whose $R$-points are
\begin{equation} \label{equation-N}
N_1(R)=	\{(x,v_1) \in \cL^\bullet \bA^m(R) \times \cL \bA^n(R) \mid \hat B(x,f(x))=t\xi(x)^2 v_1\}
\end{equation}
As the equations $\hat B(x,f(x))=t\xi(x)^2 v_1$ defining $N_1$ are essentially equivalent to the congruence equation 
\begin{equation} \label{congruence}
		\hat B(x;f(x)) \equiv 0 \mod t\xi(x)^2 
\end{equation}
one may say that $N_1$ is defined in the arc space of the affine space of $\bA^m$ by finitely many congruence equations. 

\begin{proposition} \label{Newton}
	There exists a $\nil_\infty$-quasi-isomorphism $\nu_0: N_0\to N_1$. 
\end{proposition}

\begin{proof}

We consider the Taylor expansion \eqref{Taylor} of $f(x+ t\xi(x)v_0)$ to rewrite the equation \eqref{equation-N1} in the form
\begin{equation} \label{Taylor-qt}
	f(x)+ t\xi(x) B(x;v_0) + t^2\xi(x)^2 H(x;v_0)=0.
\end{equation} 
By applying the linear form $\hat B(x)$, we get the equation
\begin{equation} \label{B-Taylor}
	\hat B(x;f(x))+ t\xi(x)^2  v_0 + t^2 \xi(x)^2 \hat B(x;H(x;v_0))=0
\end{equation}
where $\hat B$ is the adjugate matrix of $B$. We then derive the equation
\begin{equation}
	\hat B(x;f(x)) + t\xi^2(x) v_1=0
\end{equation} 
with $v_1= v_0+t\hat B(x;H(x;v_0))$. Our desired morphism $N_0\to N_1$ is given by the formula $(x,v_0) \mapsto (x,v_0+t\hat B(x;H(x;v_0))$.

We will now prove that the morphism $N_0\to N_1$ constructed above induces a bijection $N_0(R)\to N_1(R)$ for every $k$-algebra with zero infinite radical. For an arbitrary $k$-algebra $R$, and an arbitrary vector $v_1\in R[[t]]^n$, there exists a unique $v_0\in R[[t]]^n$ such that $v_1= v_0+t\hat B(x;H(x;v_0))$ by the following lemma \ref{fixed-point}. Since $v_1\in N_1(R)$ the relation \eqref{B-Taylor} holds. We derives \eqref{Taylor-qt} from Prop. \ref{uniqueness} under the assumption $\nil_\infty(R)=0$.
\end{proof}

\begin{lemma} \label{fixed-point}
Let $h:\bA^n\to \bA^n$ an arbitrary polynomial self-map. Then for every commutative ring $R$ and an arbitrary, the map $R[[t]]^n\to R[[t]]^n$ given by $v_0\mapsto v_0+th(v_0)$ is bijective. 
\end{lemma}

\begin{proof}
Let $v_1\in R[[t]]^n$. We will prove that there exists a unique $v_0$ such that $v_1=v_0+th(v_0)$. It is equivalent to say that the self-map $\phi(v)=v_1-th(v)$ of $R[[t]]^n$ admits a unique fixed point. This is just an application of the Banach fixed point theorem for contraction map.

Let us consider the norm on $R[[t]]^n$ defined as follows. For $v=\alpha_0+\alpha_1t+\alpha_2t^2+\cdots \in R[[t]]^n$ with $\alpha_i\in R^n$, we set $|v|=2^{-d}$ where $d$ is the smallest integer such that $\alpha_d\neq 0$. By construction of $R[[t]]$ as a projective limit of $R[[t]]/(t^d)$, $R[[t]]^n$ is complete metric space with respect to this norm.

Now for $h$ is polynomial, we have $|h(v)-h(w)| \leq  |v-w|$. As $\phi(v)=v_1-th(v)$, we have
$|\phi(v)-\phi(w)| \leq |v-w|/2$. In other words, the self-map $v\mapsto \phi(v)$ is contraction map with contraction ratio $\leq 1/2$. It therefore admits a unique fixed point by the Banach fixed point theorem.
\end{proof}

\section{Local structure of arc spaces}

We will simplify notations in setting $c(x)=\xi^2(x)$ and $a(x)=\hat B(x,f(x))$ in the definition of $N_1$:
\begin{equation} \label{equation-N}
N_1(R)=	\{(x,v) \in \cL^\bullet \bA^m(R) \times \cL \bA^n(R) \mid a(x)=tc(x) v\}.
\end{equation}
By definition $c$ is a polynomial function on $\bA^m$ and $a$ is a polynomial vector $a:\bA^m\to \bA^n$. 

In the goal of this section is to prove that after base change by a morphism of Weierstrass type $N\to N_1$, there exists an essentially smooth morphism $N\to Y$ where $Y$ is a scheme of finite type over the $n$-fold Cartesian product of $S_d$ over $Q_d$
\begin{equation} \label{S-d-n}
	S_d^n=S_d \times_{Q_d} S_d \times_{Q_d} \cdots \times_{Q_d} S_d
\end{equation}
where $S_d$ is the "strange" vector bundle over $Q_d$ defined in \eqref{S-d}. Recall that as $d$ varies $S_d$ form family of finite dimensional schemes, which are not of finite type, in which the first member $S_1$ is the "sawed plane". 

\begin{proposition} There exists a cover of $N_1$ by morphisms of Weiersstrass type
$N_{1,d} \to N_1$ such that for every $d$, there exists an isomorphism 
$$N_{1,d} = T_d\times \bA^\infty$$ 
where $T_d$ is a scheme of finite type over $S_{d+1}^{n+1}$.	
\end{proposition}

\begin{proof}

For every $d\in\N$, we consider the Cartesian diagram
\begin{equation}
\begin{tikzcd}
N_{1,d} \arrow{r}{\xi} \arrow{d}[swap]{\alpha_1}
&  Q_d \times \cL \Gm \times A_{d+1}^m\times \cL\bA^m
 \arrow{d}{(\beta,\alpha)} \\
N_1 \arrow{r}[swap]{\gamma}
& \cL \bA \times \cL \bA^m 
\end{tikzcd}
\end{equation}
where $\gamma$ is essentially the graph of $c$
$$\gamma(x,v)=(c(x),x)$$
and $(\beta,\alpha)$ is the morphism of Weierstrass type
\begin{equation}
	(\beta,\alpha)(q,u,\bar x,x') = (uq,\bar x+ tq x').
\end{equation}
By base change, $N_{1,d}\to N_1$ is a morphism of Weierstrass type.

Let us denote $M$ the space of quintuples $(q,u,\bar x,x',v)$:
$$M=Q_d \times \cL \Gm \times A_{d+1}^m\times \cL\bA^m \times \cL \bA^n$$
By construction $N_{1,d}$ is the closed subscheme of $M$ consisting of quintuples
\begin{equation} \label{quintuple}
	(q,u,\bar x,x',v)\in M
\end{equation}
satisfying the equations:
\begin{eqnarray}
	a(\bar x+ tq x') & = & tuqv \\
	c(\bar x+ tq x') & = & uq
\end{eqnarray}
The idea is now to expand $a(\bar x+ tq x')$ and $c(\bar x+ tq x')$ in the form
\begin{eqnarray} 
	\label{expansion-a}
	a(\bar x+tq x') &=& \bar a(q,\bar x)+ tq a'(q,\bar x,x')\\ 
	\label{expansion-c}
	c(\bar x+tq x') &=& \bar c(q,\bar x)+tq c'(q,\bar x,x')
\end{eqnarray}
where $\bar a:Q_d\times A_{d+1}^m \to Q_d \times A_{d+1}^n$ and 
$\bar c: Q_d\times \A_{d+1}^m \to Q_d \times A_{d+1}$ are deduced from $a$ and $c$ by the general procedure \ref{multijet}, $a'(q,\bar x,x')\in \cL \bA^n$ and $c'(q,\bar x,x')\in \cL \bA^1$ are respectively a $n$-vector of formal series and a single formal series whose coefficients are polynomial functions of coefficients of $q,\bar x,x'$. 

We will now construct the morphisms $\bar a:Q_d\times A_{d+1}^m \to Q_d \times A_{d+1}^n$ and $\bar c: Q_d\times \A_{d+1}^m \to Q_d \times A_{d+1}$ and the expansions \eqref{expansion-a} and \eqref{expansion-c}. 

Let $X$ be an affine $k$-scheme of finite type. The functor $X_{Q_d}$ on the category of $k$-algebras 
\begin{equation}
	X_{Q_d}(R)=\{(q,x)\mid q\in Q_d(R), x:\Spec(R[t]/q)\to X\}.
\end{equation} 
is representable by a flat affine $Q_d$-scheme $X_{Q_d}$. The fiber of $X_{Q_d}$ over separable polynomial $q\in Q_d(k)$ is isomorphic to $X^d$. The fiber of $X_{Q_d}$ over the polynomial $q=t^d$ is isomorphic to the jet space $\cL_{d-1}X$ of $X$. 
If $X=\bA^1$ is the affine line, we have
\begin{equation}
	\bA^1_{Q_d}(R)=\{(q,x)\mid q\in Q_d(R), x \in R[t]/q)\}
\end{equation} 
which is a vector bundle of rank $r$ over $Q_d$. As the canonical map $A_d(R) \to R[t]/q$ from the space $A_d(R)$ of polynomials of degree less than $d$ to $R[t]/q$ is an isomorphism by Euclidian division theorem, the latter vector bundle can be canonically trivialized
\begin{equation}\label{A1_Qd}
	Q_d\times A_d = \bA^1_{Q_d}.
\end{equation}

For every morphism $f:X\to Y$, we have an induced morphism 
\begin{equation}\label{f_Qd}
	f_{Q_d}:X_{Q_d} \to Y_{Q_d}
\end{equation}
Now the morphism $a:\bA^m\to \bA^n$ induces a morphism $a_{Q_d}:\bA^m_{Q_d} \to \bA^n_{Q_d}$
By identifying $\bA^m_{Q_d}=Q_d\times A_d^m$ and $\bA^n_{Q_d}=Q_d\times A_d^n$ as in \eqref{A1_Qd}, we get the desired morphism 
\begin{equation}
	\bar a:Q_d\times A_d^m \to Q_d\times A_d^n.
\end{equation}
Similarly, we have a morphism 
\begin{equation}
	\bar c:Q_d\times A_d^m \to Q_d\times A_d.
\end{equation}
induced by $c$.

Now $N$ can is the space of quintuples $(q,u,\bar x,x',v)$ as in \eqref{quintuple} satisfying two equations 
\begin{eqnarray}
	\bar a(q,\bar x)+ tq (a'(q,\bar x,x')-uv) & = &0 \\
	(\bar c(q,\bar x)-u_0q)+tq (c'(q,\bar x,x')-u') & = &0
\end{eqnarray}
where $u_0$ is the constant coefficient of $u$ and $u=u_0 + tu'$ where $u'$ is a formal series. 

We "solve" this system of equations by the following diagram in which all square are Cartesian:
\begin{equation*}
\begin{tikzcd}
N_{1,d} \arrow{r}{g_N} \arrow{d}
& T_d  \arrow{r} \arrow{d}
&  S_{d+1}^{n+1} \arrow{r} \arrow{d}
& Q_{d+1} \arrow{d}{0}
\\
M \arrow{r}[swap]{g_M}
& \Gm \times Q_d \times A_{d+1}^m  \times \cL\bA^{n+1} \arrow{r} \arrow{d}
& Q_{d+1}\times A_{d+1}^{n+1} \times \cL\bA^{n+1}  \arrow{r}[swap]{\alpha_{d+1}} \arrow{d}
& Q_{d+1} \times \cL \bA^{n+1}\\
& \Gm \times Q_d\times A_{d+1}^m \arrow{r}[swap]{\psi} 
& Q_{d+1} \times  A_{d+1}^{n+1}
\end{tikzcd}
\end{equation*}
We will describe the arrows in this diagram from right to left, and then move bottom up in each column
\begin{itemize}
	\item the morphism $0:Q_{d+1}\to Q_{d+1}\times \cL \bA^{n+1}$ is the zero section $q_{d+1}\mapsto (q_{d+1},0)$;
	\item $\alpha_{d+1}$ is given by the formula \eqref{alpha-d};
	\item the upper right square is Cartesian by the very definition of $S_{d+1,n+1}$;
	
	\item the morphism $\psi$ on the third line is given by $\psi(u_0,q,\bar x)=(tq,(\bar a(q,\bar x), \bar c(q,\bar x)-u_0q))$;
	\item the two down arrows from the second line to the third are obvious projections, the lower square between those lines is obviously Cartesian;
	\item $T_d$ is defined so that the upper middle square is Cartesian;
		
	\item the morphism $g_M:M\to \Gm \times Q_d \times A_{d+1}^m  \times \cL\bA^{n+1}$ maps a point $(q,u,\bar x,x',v)\in M$ with $u=u_0+tu'$
	on the point 
	$$g_M(q,u,\bar x,x',v)=(u_0,q,\bar x,(a'(q,\bar x,x')-uv,c'(q,\bar x,x')-u')).$$
\item as by definition $N=M\times_{Q_{d+1}\times\cL\bA^{n+1}} Q_{d+1}$, there exists a unique morphism $g_N:N\to S$ making the upper left square Cartesian. 
\end{itemize}

By construction, $T_d$ is a scheme of finite type over $S_{d+1}^{n+1}$ as $\psi$ is a morphism between $k$-schemes of finite type. Since the upper left corner is Cartesian, in order to prove that there exists an isomorphism $N_{1,d}=T_d\times \bA^\infty$, it is enough to prove that there exists an isomorphism 
\begin{equation} \label{iso-M}
M=(\Gm \times Q_d \times A_{d+1}^m  \times \cL\bA^{n+1})\times \bA^\infty	
\end{equation}
compatible with the projection to $\Gm \times Q_d \times A_{d+1}^m  \times \cL\bA^{n+1}$.

Given a point
$$(u_0,q,(tq,\bar x),w)\in (\Gm\times Q_d\times_{Q_{d+1}} \bA^m_{Q_{d+1}} \times \cL\bA^{n+1})(R)$$
we need to solve the equation 
$$g_M(q,u,\bar x,x',v)=(u_0,q,(tq,\bar x),w).$$
Let us write the vector $w\in \cL\bA^{n+1}(R)$ as $(w_n,w_1)$ with $w_n \in R[[t]]^n$ and $w_1\in R[[t]]$ and $u=u_0 + u't$ where $u_0\in R^\times$ and $u'\in R[[t]]$. We need to determine the set of 
$$(u',x',v)\in R[[t]]\times R[[t]]^m \times R[[t]]^n$$ 
such that 
\begin{eqnarray*}
a'(q,\bar x,x')-uv &=& w_n \\
	c'(q,\bar x,x')-u' &=& w_1 
\end{eqnarray*}
The set of solutions is uniquely parametrized by $x'\in R[[t]]^m$ because the above system of equations is equivalent to
\begin{eqnarray*}
	v & = & u^{-1}(a'(q,\bar x,x') -w_n) \\
	u' &=& c'(q,\bar x,x') -w_1.
\end{eqnarray*}
These equations yield the desired isomorphism \eqref{iso-M}.

\end{proof}

As consequence, we obtain the following diagram describing local structures of arc spaces
\begin{equation}
\begin{tikzcd}
N_{1,d}=T_d\times \bA^\infty \arrow{r}{\alpha_1} \arrow{d}[swap]{\pr_1}
&  N_1  & N_0 \arrow{l}[swap]{\nu_0} \arrow{r}{\alpha_0} & \cL X \\
T_d \arrow{r}{\gamma} & S_{d+1}^{n+1}
\end{tikzcd}
\end{equation}
where $\nu_0$ is a $\nil_\infty$-quasi-isomorphism, $\alpha_0,\alpha_1$ are morphisms of Weierstrass type, and $\gamma$ is  a morphism of finite type. We note that the positions of $\nu_0$ and $\gamma_1$ may be exchanged so that $N_{1,d}\to \cL X$ is the composition of a $\nil_\infty$-quasi-isomorphism with a morphism of Weierstrass type. On the lower line, $T_d$ is of of finite type over a the finite-dimensional scheme $S_{d+1}^{n+1}$ which is of Weierstrass type but not of finite type. It doesn't seem possible to invert orders i.e. $T_d$ is likely not being of Weierstrass type over a scheme of finite type.
 
Although, the local structure of arc spaces as described by this diagram is far from being pleasant, we can derive the existence of good slices.

\section{Construction of good slices}

We keep the notations in force in the previous section. In particular, we have a isomorphism $N_{1,d}\simeq T_d\times \bA^\infty$ where $T_d$ is a $S_{d+1}^{n+1}$-scheme of finite type. Recall that $S_{d+1}^{n+1}$, finite-dimensional but not of finite type, is a strange vector bundle over $Q_{d+1}$. We consider the subscheme $Y_d$ of $T_d$ constructed by the Cartesian diagram:
\begin{equation} \label{Y-d}
\begin{tikzcd}
Y_d \arrow{r}{} \arrow{d}[swap]{}
&  Q_{d+1}  \arrow{d}{0} \\
T_d \arrow{r}[swap]{}
& S_{d+1}^{n+1} 
\end{tikzcd}
\end{equation}
where the right vertical map is the zero section. By base change, the morphism $Y_d\to Q_{d+1}$ is of finite type and therefore $Y_d$ is a $k$-scheme affine of finite type. Let $Y_d=\Spec(R_d)$ where $R_d$ is a $k$-algebra of finite type.

For every $R_d$-point $r\in \bA^\infty(R_d)$, we consider the composed morphism 
$$\phi_1(r): Y_d \to T_d \to T_d\times \bA^\infty \simeq N_{1,d} \to N_1$$
defines a $R_d$-point of $N_1$. Since the morphism $N_0\to N_1$ is a $\nil_\infty$-quasi-isomorphism after Prop. \ref{Newton}, the induced map on $R_d$-points $N_0(R_d)\to N_1(R_d)$ is a bijection. We can thus lift $f_1(r):Y_d\to N_1$ to a morphism
$$\phi_0(r):Y_d\to N_0$$
and thus a morphism 
$$\phi(r):Y_d \to \cL X.$$

\begin{proposition}
	For every $r\in \bA^\infty(R_d)$, for every $k$-field $k'$, $y'\in Y(k')$ mapping to $x'\in \cL X(k')$, there exists an isomorphism
	$$\hat Y_{y'}\hat\times \hat D^\infty \simeq \cL X^\ehat_{x'} \hat\times \hat D^\infty.$$
\end{proposition} 

\begin{proof}
	A morphism $f:M\to N$ will be said to be a formal equivalence on formal completions if for every $m\in M(k)$ mapping to $n\in N(k)$, there exists an isomorphism 
	$$\hat M_m \hat \times \hat D^\infty \simeq \hat N_n \hat \times \hat D^\infty.$$ 
	It is clear that the composition of two morphisms which are formal equivalences on formal completions morphisms, is also a formal equivalence on formal completions. It is also clear that a morphism of Weierstrass type is a formal equivalence on formal completions. 
	
	We claim that the morphism $Y_d\to T_d$ is a formal equivalence on formal completions. 
	It follows that $\phi_1(r):Y_d \to N_1$ is a formal equivalence on formal completions for the morphism $N_{1,d}\to N_1$ is of Weierstrass type. 
	
	The morphism $N_0\to N_1$ being a $\nil_\infty$-quasi-isomorphism, induces isomorphism on formal completions of points, it follows that $\phi_0(r):Y_d\to N_0$ is a formal equivalence on formal completions. Finally the morphism $N_0\to \cL X$ is of Weierstrass type, and therefore a formal equivalence on formal completions. We derive that the morphism $\phi(r):Y_d \to \cL X$ is a formal equivalence on formal completions.   
\end{proof}

\begin{lemma}
	The morphism $Y_d\to T_d$ constructed in the diagram \eqref{Y-d} is a formal equivalence on formal completions. 
\end{lemma}

\begin{proof}
	This follows from the fact the morphism $Q_{d+1} \to S_{d+1}^{n+1}$ is a formal equivalence on formal completions. This property is preserved by base change by a morphism of finite type.
\end{proof}

\section{Acknowledgement}

This paper grew out of notes for talks given in Geometric Langlands Seminar in April 2017. I would like to thank V. Drinfeld, M. Nori and Y. Sakellaridis for useful discussions. In particular, Drinfeld helped me to clarify the statement of Prop. 3.3, and Nori pointed out  \eqref{equations-Z1} as example of a commutative ring with nonzero infinite nilradical. 

This work is partially supported by the NSF grant DMS-1702380 and the Simons foundation.


\begin{thebibliography}{5}
\providecommand{\natexlab}[1]{#1}
\providecommand{\url}[1]{\texttt{#1}}
\expandafter\ifx\csname urlstyle\endcsname\relax
  \providecommand{\doi}[1]{doi: #1}\else
  \providecommand{\doi}{doi: \begingroup \urlstyle{rm}\Url}\fi

\bibitem[Bouthier et~al.(2016)Bouthier, Ngo, and Sakellaridis]{Bouthier:2014vp}
Alexis Bouthier, Bao~Chau Ngo, and Yiannis Sakellaridis.
\newblock {On the formal arc space of a reductive monoid}.
\newblock \emph{American Journal of Mathematics}, 138.1:\penalty0 81--108,
  2016.

\bibitem[Bouthier and Kazhdan(2015)]{Bouthier:2015tk}
Alexis Bouthier and David Kazhdan.
\newblock {Faisceaux pervers sur les espaces d'arcs I: Le cas d'\'egales caract\'eristiques}.
\newblock {arXiv}:1509.02203v3, 2016.

\bibitem[Grinberg and Kazhdan(1998)]{Grinberg:1998tv}
Mikhail Grinberg and David Kazhdan.
\newblock {Versal deformations of formal arcs}.
\newblock \emph{Geometric Functional Analysis}, 10.3:\penalty0 
543-–555, 2000.

\bibitem[Drinfeld(2002)]{Drinfeld:2002vda}
Vladimir Drinfeld.
\newblock {On the Grinberg - Kazhdan formal arc theorem}.
\newblock {arXiv}:0203263v1, 2002.

\bibitem[Matsumura(1986)]{Matsumura:1986wj}
Hideyuki Matsumura.
\newblock \emph{{Commutative ring theory}}, volume~8 of \emph{Cambridge Studies
  in Advanced Mathematics}.
\newblock Cambridge University Press, Cambridge, 1986.

\end{thebibliography}
\end{document}